\documentclass[12pt, a4paper]{amsart}

\usepackage[hmargin=30mm, vmargin=25mm, includefoot, twoside]{geometry}
\usepackage[bookmarksopen=true]{hyperref}
\usepackage[normalem]{ulem}
\usepackage{amsfonts,amssymb,verbatim}
\usepackage[toc,page]{appendix}
\usepackage{latexsym}
\usepackage{xspace}
\usepackage{enumerate, paralist}
\usepackage{graphicx}
\usepackage[all]{xy}
\usepackage{extarrows}
\usepackage{tabu}
\usepackage{tikz-cd}

\usepackage{todonotes}
\usepackage{txfonts, pxfonts}

\usepackage{amsthm}
\usepackage{amsmath}

\newtheorem{thm}{Theorem}[section]
 \newtheorem{cor}[thm]{Corollary}
 \newtheorem{lem}[thm]{Lemma}
 \newtheorem{prop}[thm]{Proposition}

\newtheorem*{thmuniquemetric}{Theorem \ref{uniquemetricthm}}

\newtheorem*{thmcmacmsa}{Theorem \ref{unique metric prop}}
\newtheorem*{defcma}{Definition \ref{cmadef}}

\renewcommand\sout[1]{}

\def\jedit#1{\textcolor{black}{#1}}

\def\mu#1{\ifx#1(\mymu(\else\muii#1\fi}
\def\muii#1{\ifx#1\left\mymuii\else{\langle~\rangle}#1\fi}
\long\def\mymu(#1){\left\langle#1\right\rangle}
\long\def\mymuii#1#2){\left\langle#2\right\rangle}

 \theoremstyle{definition}
  \newtheorem{defn}[thm]{Definition}
 \theoremstyle{remark}
 \newtheorem{rem}[thm]{Remark}
  \newtheorem{ex}[thm]{Example}

 \def\diam{\mathrm{diam}\,}
 
 \def\rank{\mathrm{rank}\,}
 \def\mor{\mathrm{Mor}}
  \def\CMS{\mathcal{CMS}}
  \def\CIS{\mathcal{CIS}}
  \def\CM{\mathcal{CM}}
  \def\CI{\mathcal{CI}}
  \def\id{\mathrm{Id}}

  \def\N{\mathcal N}
  \def\I{\mathcal I}
\def\kappao{{\kappa_0}}
\def\kappav{{\kappa_5}}

\def\kappaiv{{\kappa_4}}
\def\f{\phi}
\def\g{\psi}
\def\0{\mathbf 0}
\def\1{\mathbf 1}
\def\Rp{\mathbb{R}^{+}}
\let\card=\sharp

\begin{document}

\title[Coarse median algebras]{Coarse median algebras: The intrinsic geometry of coarse median spaces and their intervals}

\author{Graham Niblo}
\author{Nick Wright}
\author{Jiawen Zhang}
\address{School of Mathematical Sciences, University of Southampton, Highfield, SO17 1BJ, United Kingdom.}
\email{g.a.niblo@soton.ac.uk}
\email{n.j.wright@soton.ac.uk}
\email{jiawen.zhang@soton.ac.uk}

\date{}
\subjclass[2010]{20F65, 20F67, 20F69}
\keywords{Coarse median spaces, median algebras, coarse interval structures and a canonical metric}

\thanks{The third author was supported for this research by a Fellowship from the Sino-British Trust, International Exchanges 2017 Cost Share (China) grant EC$\backslash$NSFC$\backslash$170341, NSFC11871342 and NSFC11811530291.}
\baselineskip=16pt

\begin{abstract}
This paper establishes a new combinatorial framework for the study of coarse median spaces, bridging the worlds of asymptotic geometry, algebra and combinatorics. We introduce a simple and entirely algebraic notion of coarse median algebra which simultaneously generalises the concepts of bounded geometry coarse median spaces and  classical discrete median algebras.
We  study the coarse median universe from the perspective of  intervals, with a particular focus on cardinality as a proxy for distance. 
 In particular we prove that the metric on a quasi-geodesic coarse median space of bounded geometry can be constructed up to quasi-isometry using only the coarse median operator.
Finally we develop a concept of rank for coarse median algebras in terms of the geometry of intervals and show that the notion of finite rank coarse median algebra provides a natural higher dimensional analogue of Gromov's concept of $\delta$-hyperbolicity.
\end{abstract}

\maketitle

\section{Introduction}
Gromov's notion of a CAT(0) cube complex has played a significant role in  major results in topology, geometry and group theory. Its power stems from the beautiful interplay between the non-positively curved geometry of the space and the median algebra structure supported on the vertices as outlined by Roller \cite{roller1998poc}. Coarse median spaces as introduced by Bowditch \cite{bowditch2013coarse} provide a geometric coarsening of  CAT(0) cube complexes which additionally  includes $\delta$-hyperbolic spaces, mapping class groups and hierarchically hyperbolic groups \cite{behrstock2017hierarchically,behrstock2015hierarchically}.

The interaction between the geometry and combinatorics of a CAT(0) cube complex is mediated by the fact that the edge metric can be computed entirely in terms of the median. In contrast, for a coarse median space the metric is an essential part of the data, as evidenced by the fact that almost any ternary algebra can be made into a coarse median space by equipping it with a bounded metric. This 	prompts the question to what extent there could be a combinatorial  characterisation of coarse medians mirroring the notion of a median algebra. We will provide  the missing combinatorial framework by defining \emph{coarse median algebras}.

\subsection{Bowditch's definition of coarse median space}
\begin{defn}[Bowditch, \cite{bowditch2013coarse}]\label{Bowditch original def}
A \emph{coarse median space} is a triple $(X,d,\mu)$, where $(X,d)$ is a metric space and $\mu$ is a ternary operator on $X$ satisfying the following:

\begin{itemize}
\item[(M1)] For all $a,b\in X$,  $\mu(a,a,b)=a$;
\item[(M2)] For all $a,b,c\in X$, $\mu(a,b,c)=\mu(a,c,b)=\mu(b,a,c)$;
\item[(B1)\phantom{'}] There are constants $k, h(0)$ such that for all $a,b,c,a',b',c'\in X$ we have
\[
d(\mu(a,b,c),\mu(a',b',c'))\leq k\left(d(a,a')+d(b,b')+d(c,c')\right) + h(0);
\]
\item[(B2)\phantom{'}] There is a function $h:\mathbb N\rightarrow \Rp$ with the following property. Suppose that $A\subseteq X$ with $1\leq |A| \leq p < \infty$, then there is a finite median algebra $(\Pi, \mu_\Pi)$ and maps $\pi:A\rightarrow \Pi$ and $\lambda:\Pi\rightarrow X$ such that for all $x,y,z\in \Pi$ we have
\[
d\big(\lambda(\langle x,y,z\rangle_\Pi), \langle{\lambda(x)}, {\lambda(y)}, {\lambda(z)}\rangle\big) \leq h(p),
\]
and for all $a \in A$ we have
\[
d(a, \lambda\pi(a))\leq h(p).
\]
\end{itemize}
\end{defn}

The metric plays the crucial role of measuring  and controlling the extent to which the  ternary operator (referred to as \emph{the coarse median}) approximates a classical median operator. Our observation is that the additional metric data can be replaced by the structure of the intervals in the space which are intrinsic to the median operator: the cardinality of intervals serves as a proxy for distance.\footnote{This is perhaps counterintuitive: firstly because interval cardinality is far from being a metric, and secondly because even in a geodesic coarse median space the geodesic between two points can lie well outside the corresponding interval (see \cite[Theorem 5.1]{niblo2017four}).}

\subsection{Coarse median algebras} We now define the new notion of coarse median algebra as the algebraic parallel of coarse median spaces.

Recall that a \emph{ternary algebra} is a set $X$ equipped with a function $\mu:X^3\rightarrow X$ where $\mu(x,y,z)$ denotes the value at $(x,y,z)$.

\begin{defn}\label{interval}
Let $(X,\mu)$ be a ternary algebra. For any $a,b\in X$, \emph{the interval $[a,b]$} is the set $\{\mu(a,x,b)\mid x\in X\}$. We say that $(X,\mu)$ has \emph{finite intervals} if for every $a,b\in X$ the interval $[a,b]$ is a finite set.
\end{defn}

A \emph{discrete} median algebra, which is familiar to geometric group theorists as the vertex set of a CAT(0) cube complex, is simply a median algebra with finite intervals (see for example, \cite{roller1998poc}).

\begin{defn}\label{cmadef}
A  \emph{coarse median algebra} is a ternary algebra $(X,\mu)$ with finite intervals such that:
\begin{itemize}
\item[(M1)\phantom{'}] For all $a,b\in X$,  $\mu(a,a,b)=a$;
\item[(M2)\phantom{'}] For all $a,b,c\in X$, $\mu(a,b,c)=\mu(a,c,b)=\mu(b,a,c)$;
\item[(M3)'] There exists a constant $K\geq 0$ such that for all $a,b,c,d,e\in X$ the cardinality of the interval
\;$\big[\mu(a,b,\mu(c,d,e)),\, \mu(\mu(a,b,c),\mu(a,b,d),e)\big]$\; is at most $K$.
\end{itemize}
\end{defn}

\begin{figure}[htbp] 
   \centering
   \includegraphics[width=6.5in]{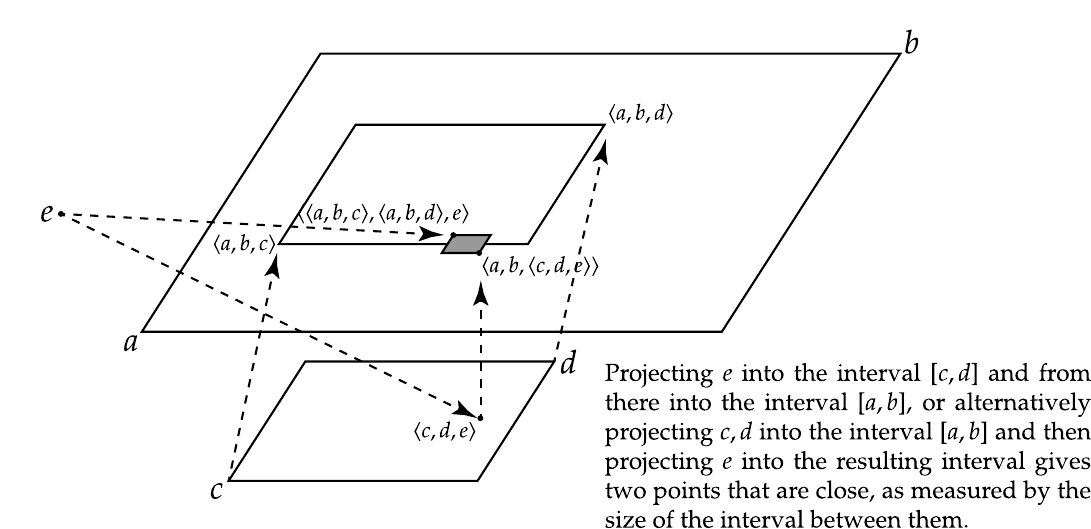}
   \caption{An illustration of the meaning of condition (M3)'}
   \label{fig:example}
\end{figure}
Putting $K=1$ in the definition reduces (M3)' to the classical 5-point condition $\mu(a,b,\mu(c,d,e))= \mu(\mu(a,b,c),\mu(a,b,d),e)$ defining a median algebra, so Definition \ref{cmadef}  generalises the notion of discrete median algebra.

\subsection{The induced metric}
At first sight the data defining a coarse median algebra appears to carry a lot less information than Bowditch's coarse median spaces. However as we will see in Section \ref{coarsemedianalgebras}, the finite intervals condition will allow us to define a metric $d_\mu$ on the ternary algebra purely in terms of the operator $\mu$. Moreover any bounded geometry\footnote{Throughout the paper, we only consider the notion of bounded geometry in the setting of discrete metric spaces. See Definition \ref{defn:basic metric}(3).} coarse median space is a coarse median algebra. Indeed generalising the notion of bounded valency for a graph (see Definition \ref{bounded valency def}), we have the following equivalence:

\newcommand{\cmacmsthma}{
Let $(X,\mu)$ be a bounded valency ternary algebra. The following are equivalent:
\begin{enumerate}
\item $(X,\mu)$ is a coarse median algebra;
\item $(X,d_\mu, \mu)$ is a coarse median space;
\item There exists a metric $d$ such that $(X,d, \mu)$ is a bounded geometry coarse median space.
\end{enumerate}
}
\begin{thm}\label{unique metric prop}
 \cmacmsthma
\end{thm}

As an application of these ideas we show that for any bounded geometry quasi-geodesic coarse median space, the metric is uniquely determined by the coarse median operator up to quasi-isometry.

\begin{thm}\label{uniquemetricthm}\label{bi-lip equi}
For a bounded geometry quasi-geodesic coarse median space $(X,d,\mu)$, the metric $d$ is unique up to quasi-isometry.
\end{thm}

Indeed within this equivalence class of metrics there is a canonical representative $d_\mu$ defined purely in terms of the  coarse median operator $\mu$ (see Theorem \ref{d qi to d_mu}).

This theorem fails without the quasi-geodesic condition as shown by Example \ref{F infty}, but the failure in this example is suggestive. It is interesting to speculate to what extent the non-uniqueness could be described.

\subsection{Rank}
As well as providing a relatively simple characterisation of a coarse median operator, our combinatorial approach introduces a new perspective on the notion of rank in the coarse median world. We provide three new ways to characterise rank each of which  is a higher rank analogue of one of the classical characterisations of Gromov's $\delta$-hyperbolicity:

\begin{table}[h]
\begin{center}
\tabulinesep=1mm
\begin{tabu} to 1
\textwidth{|X[1.8,c]|X[3.2,c]|}
\hline
\textbf{Hyperbolic spaces} & \textbf{Coarse median spaces/algebras of rank $n$} \\  \hline
approximating finite subsets by trees & approximating finite subsets by CAT(0) cube complexes of dimension $n$ \cite{bowditch2013coarse} \\  \hline
Gromov's inner product (``thin squares'') condition & thin $(n+1)$-cubes condition: Theorem \ref{hyper rank} (3) and Lemma \ref{cma rank lemma} \\  \hline
\strut slim triangle condition & $(n+1)$-multi-median condition: Theorem \ref{hyper rank} (2) \\  \hline
pencils of quasi-geodesics grow linearly & interval growth is $o(n+1)$: Theorem \ref{growth rank} \\ \hline
\end{tabu}
\end{center}
\end{table}

\bigskip

\newcommand{\hyperrankthm}{Let $(X,d,\mu)$ be a coarse median space and $n \in \mathbb N\setminus \{0\}$, then the following are equivalent:
\begin{enumerate}[1)]
  \item $\rank X \leqslant n$;
  \item \emph{Multi-median condition:} There exists a non-decreasing function $\psi$ such that for any $\lambda>0$ and any $x_1,\ldots,x_{n+1},q \in X$, we have
      $$\bigcap_{i \neq j} \N_\lambda([x_i,x_j]) \subseteq \bigcup_{i=1}^{n+1} \N_{\psi(\lambda)}([x_i,q]);$$
  \item \emph{Thin $(n\!+\!1)$-cubes condition:} There exists a non-decreasing function $\varphi$, such that
      \begin{equation*}
        \min\{d(p,\langle x_i,p,q\rangle):i=1,\ldots,n+1\} \leqslant \varphi(\max\{d(p,\langle x_i,x_j,p\rangle): i\neq j\})
      \end{equation*}
      for any $x_1,\ldots,x_{n+1}$ and $p,q \in X$.
\end{enumerate}}

\newcommand{\growthrankthm}{Let $(X,d,\mu)$ be a uniformly discrete, quasi-geodesic coarse median space with bounded geometry and $n$ be a natural number. The following are equivalent:
\begin{enumerate}
\item $(X,d,\mu)$ has rank at most $n$;
\item there is a function $p: \Rp\to \Rp$ with $p(r)=o(r^{n+\epsilon})$ for all $\epsilon>0$, such that $\card~ [a,b] \leqslant p(d(a,b))$ for any $a,b\in X$;
\item there is a function $p: \Rp \to \Rp$ with $p(r)/r^{n+1}\stackrel{r\rightarrow\infty}{\longrightarrow}0$, such that $\card~ [a,b] \leqslant p(d(a,b))$ for any $a,b\in X$.
\end{enumerate}}

The thin $(n+1)$-cubes condition reduces, in the case of $n=1$, to the existence of a non-decreasing function $\varphi$ such that for all $a,b,c$ and $p$ we have
\begin{equation*}
  \min\{d(p,\mu(a,b,p)),d(p,\mu(b,c,p))\} \leqslant \varphi(d(p,\mu(a,c,p))).
\end{equation*}
For geodesic coarse median spaces, this is a characterisation of hyperbolicity (see Section \ref{generalised hyperbolicity}).

The above inequality has the virtue that it is quasi-isometry invariant: the disadvantage of Gromov's 4-point condition, when applied to non-geodesic spaces, is that it is not. Hence the class of quasi-geodesic coarse median spaces satisfying our variant of the 4-point condition is closed under quasi-isometry, so we propose the class of rank 1 coarse median algebras as a more robust generalisation of hyperbolicity beyond the (quasi)-geodesic world.

\subsection{Outline of the paper}
The paper is organised as follows. In Section \ref{preliminaries}, we recall background definitions including coarse median spaces, their ranks and \v{S}pakula \& Wright's notion of iterated coarse median operators from \cite{niblo2017four,vspakula2017coarse}.

In Section \ref{coarseintervalstructures}, by analogy with Sholander's results for median algebras and interval structures \cite{sholander1954medians}, we give a characterisation of coarse median spaces entirely in terms of their intervals.

In Section \ref{growth section}, we introduce and study characterisations of rank in the context of coarse interval structures and show that the correspondences from Section \ref{coarseintervalstructures} preserve rank for coarse median spaces.

In Section \ref{coarsemedianalgebras}, we study the intrinsic metric on a ternary algebra and show that it is unique up to quasi-isometry for any quasi-geodesic coarse median space of bounded geometry. Motivated by this in Section \ref{coarsemedianalg}, we study the geometry of coarse median algebras.  We establish that these simultaneously generalise the notions of:

\begin{enumerate}
\item Classical discrete median algebras;
\item Geodesic hyperbolic spaces of bounded geometry;
\item Bounded geometry coarse median spaces.
\end{enumerate}

The correspondences established in this paper can also be couched as correspondences between (or equivalences of) suitable categories. In the Appendix we examine the notion of morphisms and the definitions of the functors required by that approach.

\subsection*{Acknowledgements}
We would like to thank the anonymous referee for many useful suggestions to make the paper more readable.

\section{Preliminaries}\label{preliminaries}
We follow the conventions established in \cite{niblo2017four}.

\subsection{Metrics and geodesics}

\begin{defn}\label{defn:basic metric}
Let $(X,d)$ be a metric space.
\begin{enumerate}
  \item A subset $A \subseteq X$ is \emph{bounded}, if its diameter $\diam (A):=\sup\{d(x,y):x,y\in A\}$ is finite; $A$ is a \emph{net} in $X$, if there exists some constant $C>0$ such that for any $x\in X$, there exists some $a\in A$ such that $d(a,x) \leqslant C$.
  \item The metric space $(X,d)$ is said to be \emph{uniformly discrete} if there exists a constant $C>0$ such that for any $x \neq y \in X$, $d(x,y)>C$.
  \item The metric space $(X,d)$ is said to have \emph{bounded geometry} if, for any $r>0$, there exists some constant $n \in \mathbb N$ such that $\card~B(x,r) \leqslant n$ for any $x\in X$.
  \item Points $x,y\in X$ are said to be \emph{$s$-close} (with respect to the metric $d$) if $d(x,y)\leqslant s$. If $x$ is $s$-close to $y$, we write $x\thicksim_s y$. Maps $f,g:X\to Y$ are said to be \emph{$s$-close}, written $f\thicksim_s g$, if $f(x)\thicksim_s g(x)$ for all $x\in X$.
\end{enumerate}
\end{defn}

\begin{defn} Let $(X,d), (Y,d')$ be metric spaces and $L,C>0$ be constants.
\begin{enumerate}
  \item An \emph{$(L,C)$-large scale Lipschitz map} from $(X,d)$ to $(Y,d')$ is a map $f:X\rightarrow Y$ such that for any $x,x'\in X$, $d'(f(x),f(x'))\leqslant Ld(x,x')+C$.
  \item An \emph{$(L,C)$-quasi-isometric embedding} from $(X,d)$ to $(Y,d')$ is a map $f:X\rightarrow Y$ such that for any $x,x'\in X$, $L^{-1} d(x,x')-C\leqslant d'(f(x),f(x'))\leqslant Ld(x,x')+C$.
   \item An \emph{$(L,C)$-quasi-isometry} from $(X,d)$ to $(Y, d')$ is an $(L,C)$-large scale Lipschitz map $f:X\rightarrow Y$ such that there exists another $(L,C)$-large scale Lipschitz map $g: Y \to X$ with $f\circ g\thicksim_C \id_Y$ and $g\circ f\thicksim_C \id_X$.
  \item $(X,d)$ is said to be \emph{$(L,C)$-quasi-geodesic}, if for any two points $x,y\in X$ there exists an $(L,C)$-quasi-isometric embedding of the interval $[0, d(x,y)]$ into $X$ taking the endpoints to $x,y$ respectively.
      If we do not care about the constant $C$ we say that $(X,d)$ is \emph{$L$-quasi-geodesic}. If $(X,d)$ is $(1,0)$-quasi-geodesic then we say that $X$ is \emph{geodesic}.
\end{enumerate}
We will take the liberty of omitting the  parameters $L,C$ where their values are not germane to the discussion.
\end{defn}

\begin{defn}
Let $(X,d), (Y,d')$ be metric spaces, $\rho: \Rp \to \Rp$ a proper function and $C>0$ a constant.
\begin{enumerate}
  \item A \emph{$\rho$-bornologous map} from $(X,d)$ to $(Y,d')$ is a function $f:X\rightarrow Y$ such that for all $x,x'\in X$, $d'(f(x), f(x')) \leqslant \rho(d(x,x'))$.
   \item $f$ is \emph{proper} if given any bounded subset $B \subseteq Y$, $f^{-1}(B)$ is bounded.
  \item A \emph{$\rho$-coarse map} from $(X,d)$ to $(Y,d')$ is a proper $\rho$-bornologous map.
  \item A \emph{$(\rho,C)$-coarse equivalence} from $(X,d)$ to $(Y, d')$ is a $\rho$-coarse map $f:X\rightarrow Y$ such that there exists another $\rho$-coarse map $g: Y \to X$ with $f\circ g\thicksim_C \id_Y$ and $g\circ f\thicksim_C \id_X$. In this case, $g$ is called a $(\rho,C)$-\emph{coarse inverse} of $f$.
\end{enumerate}
When the parameters  $\rho,C$  are not germane to the discussion we omit them.
\end{defn}

\subsection{Median Algebras}\label{medalg}
As discussed in \cite{bandelt1983median} there are a number of equivalent formulations of the axioms for median algebras. We will use the following formulation from \cite{birkhoff1947ternary}:

\begin{defn}\label{median algebra defn}
Let $X$ be a set and $\mu$ a ternary operation on $X$. Then $\mu$ is a \emph{median operator} and the pair $(X,\mu)$ is a \emph{median algebra} if the following are satisfied:

\begin{itemize}
  \item[(M1) Localisation:] $\mu(a,a,b)=a$ for all $a,b\in X$;
  \item[(M2) Symmetry:] $\mu(a_1,a_2,a_3)=\langle a_{\sigma(1)},a_{\sigma(2)},a_{\sigma(3)}\rangle$ for all $a_1,a_2,a_3\in X$ and permutation $\sigma$ of $\{1,2,3\}$;
  \item[(M3) The 5-point condition:]  $\mu(a,b,\mu(c,d,e))= \mu(\mu(a,b,c),\mu(a,b,d),e)$ for all $a,b,c,d,e\in X$.
\end{itemize}
\end{defn}

Axiom (M3) is equivalent to the $4$-point condition given in \cite{kolibiar1974question}, see also \cite{bandelt2008metric}:
\begin{equation}\label{four point}
\mu(\mu(a,b,c),b,d)=\mu(a,b, \mu(c,b,d)).
\end{equation}
This can be viewed as an associativity axiom: For each $b\in X$ the binary operator
\[(a,c)\mapsto a*_b c:=\mu(a,b,c)\]
is associative. Note that this binary operator is also commutative by (M2).

\begin{ex}\label{mediancube}
An important example is furnished by the \emph{median $n$-cube}, denoted by $I^n$, which is the $n$-dimensional vector space over $\mathbb Z_2$ with the median operator $\mu_n$ given by  majority vote on each coordinate. More generally as remarked in the introduction, a discrete median algebra is one in which the intervals $\{\mu(a,x,b)\mid x\in X\}$ are finite. These algebras are precisely the ones that arise as the vertex sets of CAT(0) cube complexes.
\end{ex}

\subsection{Coarse median spaces}\label{def of cms}

In \cite{niblo2017four} we showed how to replace Bowditch's original definition of a coarse median space (Definition \ref{Bowditch original def}), involving $n$-point approximations for all $n$, in terms of a $4$-point condition mirroring the classical $4$-point condition for median algebras. This may also be viewed as an analogue of Gromov's $4$-point condition for hyperbolicity.

\begin{prop}[Theorem 3.1, \cite{niblo2017four}]\label{def:coarse median space}
A triple $(X,d,\mu)$ is a coarse median space if the pair $(X,d)$ is a metric space and $\mu$ is a ternary operator satisfying axioms (M1), (M2) together with the following:
\begin{itemize}
 \item[(C1) {\rm Affine control:}] There exists an affine function $\rho:[0,+\infty)\to [0,+\infty)$ such that for all $a,a',b,c\in X$, we have
      $$d(\mu(a,b,c), \mu(a',b,c)) \leqslant \rho(d(a,a'));$$
  \item[(C2) \emph{Coarse 4-point condition}:] There exists a constant $\kappaiv>0$ such that for any $a,b,c,d\in X$, we have
\[
\mu(\mu(a,b,c),b,d) \thicksim_{\kappaiv} \mu(a,b, \mu(c,b,d)).
\]
 \end{itemize}
\end{prop}

In the same way that axiom (M3) for a median algebra is equivalent to the $4$-point condition (\ref{four point}), Bowditch's condition (B2) for a coarse median space ensures that there exists a constant $\kappav>0$ such that for any five points $a,b,c,d,e\in X$ we have
\begin{equation}\label{five point estimate}
\mu(a,b,\mu(c,d,e)) \thicksim_\kappav \mu(\mu(a,b,c),\mu(a,b,d),e).
\end{equation}
By Proposition \ref{def:coarse median space}, the constant $\kappav$ depends only on the parameters $\rho$ and $\kappaiv$. However it is convenient to carry it with us in calculations. With this in mind we make the following definition.

\begin{defn}
We define the notion of \emph{parameters} for a coarse median space $(X,d,\mu)$ to be any 3-tuple $(\rho, \kappaiv, \kappav)$ of constants satisfying the axioms in Definition \ref{coarse median operator} together with estimate (\ref{five point estimate}). Writing the control function $\rho$ in the form of $\rho(t)=Kt+H_0$ for some positive constants $K$ and $H_0$, we also refer to the 4-tuple $(K,H_0,\kappaiv, \kappav)$ as parameters of $(X,d,\mu)$.
\end{defn}

As remarked by  Bowditch \cite{bowditch2013coarse} (in the discussion following Lemma 8.1 there), one can relax axioms (M1) and (M2) without loss to the following:
\begin{itemize}
  \item[(C0) \emph{Coarse localisation and coarse symmetry}:] There is a constant $\kappao >0$ such that for all points $a_1,a_2,a_3$ in $X$, $\mu(a_1,a_1,a_2)\thicksim_{\kappao} a_1$ and $\langle a_{\sigma(1)},a_{\sigma(2)},a_{\sigma(3)}\rangle \thicksim_{\kappao} \mu(a_1,a_2,a_3)$ for any permutation $\sigma$ of $\{1,2,3\}$.
\end{itemize}
Any such ternary operator can be replaced by an operator  satisfying the localisation and symmetry conditions (M1) and (M2) of Definition \ref{median algebra defn} at the cost of moving the values $\mu(a,b,c)$ only a uniformly bounded distance. These axioms are more robust under coarse constructions so we make the following definition:

\begin{defn}\label{coarse median operator}
A \emph{coarse median structure} on a metric space $(X,d)$ is a triple $(X,d,\mu)$ satisfying axioms (C0), (C1) and (C2). Parameters for the structure are given by the function $\rho$ from (C1) together with the constants $\kappao, \kappaiv, \kappav$.
\end{defn}

\begin{rem}\label{unifclosestructure}
The discussion above can be summarised as the assertion that any coarse median structure $(X,d,\mu)$ can be replaced by a coarse median space $(X,d,\mu')$ such that the maps $\mu$ and $\mu'$ are uniformly close. Abusing terminology, we say that the space is uniformly close to the structure.
\end{rem}

\subsection{Rank for a coarse median space}
As in the case of median algebras, there is a notion of \emph{rank} for a coarse median space. In terms of Bowditch's original definition of coarse medians, the rank is simply the least upper bound on the ranks of the required approximating median algebras given by condition (B2).

Using the formulation of coarse median given in Definition \ref{coarse median operator} (which only indirectly implies the existence of approximations for all finite subsets by median algebras), a characterisation of ranks in terms of suitable embeddings of cubes is more useful.

\begin{defn}
For a ternary algebra $(X,\mu_X)$ and a coarse median space $(Y,d_Y,\mu_Y)$, a map $f:X \to Y$ is said to be a \emph{$C$-quasi-morphism} for some $C>0$ if for $a,b,c\in X$, we have $\langle f(a),f(b),f(c)\rangle_Y\thicksim_C f(\mu(a,b,c)_X)$.
\end{defn}

\begin{prop}[Theorem 4.11, \cite{niblo2017four}]\label{char for high rank-final}
Let $(X,d,\mu)$ be a coarse median space and $n\in \mathbb{N}$. Then the following conditions are equivalent.
\begin{enumerate}
  \item $\rank X \leqslant n$;
  \item For any $\lambda>0$, there exists a constant $C=C(\lambda)$ such that for any $a,b\in X$ and $e_1,\ldots,e_{n+1}\in[a,b]$ with $\langle e_i,a,e_j\rangle\thicksim_\lambda a$ ($i\neq j$), one of the points $e_i$ is $C$-close to $a$;
  \item For any $L>0$, there exists a constant $C=C(L)$ such that for any $L$-quasi-morphism $\sigma$ from the median cube $I^{n+1}$ to $X$, the image  $\sigma(\bar{e}_i)$ of one of the cube vertices $\bar{e}_i$ adjacent to the origin $\bar \0$ is $C$-close to the image $\sigma(\bar \0)$.
\end{enumerate}
\end{prop}

We note that while part (3) of this theorem is slightly different from that stated in \cite[Theorem 4.11]{niblo2017four}, the given proof establishes this version too.

We also need the following notion of coarse median isomorphisms when we characterise rank via interval growths in Section \ref{growth section}.
\begin{defn}\label{cms isom}
Let $(X,d_X)$, $(Y,d_Y)$ be metric spaces and $\mu_X$, $\mu_Y$ be coarse medians on them, respectively. A map $f:X \to Y$ is called a $(\rho,C)$-\emph{coarse median isomorphism} for some proper function $\rho: \mathbb{R}^+ \to \mathbb{R}^+$ and constant $C>0$, if $f$ is a $(\rho,C)$-coarse equivalence as well as a $C$-quasi-morphism.
\end{defn}

There is a nice categoric explanation of this terminology given in Appendix \ref{The coarse median (space) category}. We will show in Remark \ref{dep of para for coarse median iso} that for a $(\rho_+,C)$-coarse median isomorphism $f$, any $(\rho_+,C)$-coarse inverse $g$ is a $C'$-quasi-morphism with the constant $C'$ depending only on $\rho_+,C$ and parameters of $X,Y$. In this case, we will also refer to $g$ as an \emph{inverse} of $f$.

\subsection{Iterated coarse medians}
We recall the following definition from \cite{vspakula2017coarse}:
\begin{defn}\label{coarseiteratedmediandefn}
Let $(X,d,\mu)$ be a coarse median space and $b\in X$. For $x_1\in X$, define:
$$\mu(x_1;b):=x_1.$$
For $k \geqslant 1$ and $x_1,\ldots,x_{k+1} \in X$, define the \emph{coarse iterated median}
$$\mu(x_1,\ldots,x_{k+1};b):=\mu(\mu(x_1,\ldots,x_k;b),x_{k+1},b).$$
Note that this definition ``agrees'' with the original coarse median operator $\mu$ in the sense that for any $a,b,c$ in $X$, we have $\mu(a,b,c)=\mu(a,b;c)$.
\end{defn}

It was established in \cite[Section 5]{vspakula2017coarse} that in a median algebra, the iterated median $m:=\mu(x_1,\ldots,x_{k+1};b)$ is characterised by the fact that the interval $[m, b]$ is the intersection of the intervals $[x_i,b]$ for $i=1,\ldots, k+1$.

Fixing a point $b$ and (as in Section \ref{medalg}) writing the coarse median $\mu(x_1,x_2,b)$ as $x_1\mathop*_b x_2$, the iterated median $\mu(x_1,\ldots,  x_k;b)$ can be written as $( (x_1\mathop*_b x_2)\mathop*_bx_3)\mathop*_b\ldots \mathop*_bx_k$. In this notation the four point condition is, precisely the statement that $(x_1\mathop*_b x_2)\mathop*_bx_3$ is uniformly close to the product $x_1\mathop*_b( x_2\mathop*_bx_3)$. This along with the commutativity of the operation $\mathop*_b$ allows the rearrangment of iterated medians. See Lemma \ref{coarse iterated estimate new1} below.

\bigskip

In \cite[Lemmas 2.16--2.19]{niblo2017four} we established the following estimates:

\begin{lem}\label{coarse iterated estimate}
Let $(X,d,\mu)$ be a coarse median space with parameters $(\rho, \kappaiv, \kappav)$. Then there exist non-decreasing functions $\rho_n, H_n:\Rp\to \Rp$ and constants $C_n, D_n$ depending only on $\rho, \kappaiv, \kappav$ such that for any $a, a_0,a_1,\ldots,a_n$ and $b, b_0,b_1,\ldots,b_n \in X$ we have:
\begin{enumerate}
\item $d(\mu(a_1,\ldots,a_n;a_0),\mu(b_1,\ldots,b_n;b_0)) \leqslant \rho_n (\sum_{k=0}^n d(a_k,b_k))$.
\item Let $(\Pi,\mu_\Pi)$ be a median algebra and $\sigma \colon \Pi \rightarrow X$ an $L$-quasi-morphism. For any $x_1,\ldots,x_n,b\in\Pi$, we have
$$\sigma(\langle x_1,\ldots,x_n;b\rangle_\Pi)\thicksim_{H_n(L)}\langle\sigma(x_1),\ldots,\sigma(x_n);\sigma(b)\rangle.$$
\item $\mu(a,b,\mu(a_1,\ldots,a_{n-1};a_n))\thicksim_{C_n}\mu(\mu(a,b,a_1),\ldots,\mu(a,b,a_{n-1});a_n)$.
\item $\mu(a,b,\mu(a_1,\ldots,a_{n-1};a_n))\thicksim_{D_n}\mu(\mu(a,b,a_1),\ldots,\mu(a,b,a_{n-1});\mu(a,b,a_n))$.
\end{enumerate}
\end{lem}

Here we provide additional estimates that will give us the control we need later to analyse the structure of coarse cubes in Section \ref{growth}.

\begin{lem}\label{coarse iterated estimate new1}
Let $(X,d,\mu)$ be a coarse median space with parameters $(\rho, \kappaiv, \kappav)$. Then for any $n\in \mathbb N$, there exists a constant $G_n$ depending only on $\rho, \kappaiv, \kappav$ such that for any $a_1,\ldots,a_n,b\in X$ and any permutation $\sigma\in S_n$, we have
$$\langle a_{\sigma(1)},\ldots,a_{\sigma(n)};b\rangle \thicksim_{G_n} \mu(a_1,\ldots,a_n;b).$$
\end{lem}

\begin{proof}
We proceed by induction on $n$. When $n=1$ or $2$, we may take $G_1=G_2=0$ by definition and axiom (M2).

Now assume that the result holds for $1,2,\dots,n-1$ and we consider the case of $n$. As usual it is sufficient to prove the lemma when $\sigma$ is a transposition of the form $(1j)$. If $j<n$ then by definition, we have
$$\mu( a_1,\ldots, a_n; b )=\langle \langle a_1,\ldots,a_j;b\rangle,a_{j+1},\ldots,a_n;b \rangle.$$
Inductively $\langle a_1,\ldots,a_j;b\rangle \thicksim_{G_j} \langle a_j,a_2,\ldots,a_{j-1},a_1;b\rangle$ and the result follows by Lemma \ref{coarse iterated estimate} (1). It remains to check the case $\sigma=(1n)$. By the inductive step, we have
\begin{eqnarray*}
&&\mu(a_n,a_2,\ldots, a_{n-1},a_1;b)=\mu(\mu(a_n,a_2,\ldots,a_{n-1};b), a_1, b )\\
&\thicksim_{\rho(G_{n-1})}& \mu(\mu(a_2,\ldots,a_{n-1},a_n;b), a_1, b )= \mu(  \mu(\mu(a_2,\ldots,a_{n-1};b), a_n , b), a_1, b )\\
&\thicksim_{\kappaiv}&\mu( \mu(\mu(a_2,\ldots,a_{n-1};b),a_1,b),a_n,b )=\mu( \mu(a_2,\ldots,a_{n-1},a_1;b),a_n,b )\\
&\thicksim_{\rho(G_{n-1})}& \mu( \mu(a_1,a_2,\ldots,a_{n-1};b),a_n,b )= \mu(a_1,a_2,\ldots,a_n;b).
\end{eqnarray*}
Hence for the transposition $(1n)$, we have
$$\mu(a_n,a_2,\ldots, a_{n-1},a_1;b) \thicksim_{2\rho(G_{n-1})+\kappaiv}\mu(a_1,a_2,\ldots,a_n;b).$$
This completes the proof.
\end{proof}

\begin{lem}\label{coarse iterated estimate new2}
Let $(X,d,\mu)$ be a coarse median space with parameters $(\rho, \kappaiv, \kappav)$. Then for any $n$, there exists a constant $E_n$ depending only on $\rho, \kappaiv, \kappav$ such that for any $1\leqslant k\leqslant n$ and $a_1,\ldots,a_n,b \in X$, we have
\begin{equation}\label{EQ3}
\mu(a_1,\ldots,a_k;\mu(a_1,\ldots,a_n;b))\thicksim_{E_n}\mu(a_1,\ldots,a_k;b).
\end{equation}
\end{lem}

\begin{proof}
Fix an $n$. By Axiom (B2), there exists a constant $h_{n+1}>0$ such that for any $a_1,\ldots,a_n,b \in X$ there exist a finite median algebra $(\Pi,\mu_\Pi)$, points $\bar{a}_1, \ldots, \bar{a}_n, \bar{b} \in \Pi$ and an $h_{n+1}$-quasi-morphism $\lambda: \Pi \rightarrow X$ satisfying $\lambda(\bar{a}_i) \thicksim_{h_{n+1}} a_i$ for $i=1,\ldots, n$ and $\lambda(\bar{b}) \thicksim_{h_{n+1}} b$. From parts (1) and (2) of Lemma \ref{coarse iterated estimate} with the control functions $H_n$ and $\rho_n$ therein, we have
\[
\mu(a_1,\ldots,a_n;b)\thicksim_{\rho_n((n+1)h_{n+1})} \langle\lambda(\bar{a}_1),\ldots,\lambda(\bar{a}_n);\lambda(\bar{b})\rangle \thicksim_{H_n(h_{n+1})} \lambda(\langle\bar{a}_1,\ldots,\bar{a}_n;\bar{b}\rangle_{\Pi}).
\]
Similarly for any $1\leqslant k\leqslant n$, we have
\[
\mu(a_1,\ldots,a_k;b)\thicksim_{\rho_n((n+1)h_{n+1}) + H_n(h_{n+1})} \lambda(\langle\bar{a}_1,\ldots,\bar{a}_k;\bar{b}\rangle_\Pi)
\]
and
\begin{align*}
\mu(a_1,\ldots,a_k;\mu(a_1,\ldots,a_n;b)) &\thicksim_{\rho_n(nh_{n+1}+\rho_n((n+1)h_{n+1}) + H_n(h_{n+1}))} \langle\lambda(\bar{a}_1),\ldots,\lambda(\bar{a}_k); \lambda(\langle\bar{a}_1,\ldots,\bar{a}_n;\bar{b}\rangle_\Pi)\rangle \\
& \thicksim_{H_n(h_{n+1})} \lambda(\langle \bar{a}_1,\ldots,\bar{a}_k;\langle\bar{a}_1,\ldots,\bar{a}_n;\bar{b}\rangle_\Pi \rangle_\Pi).
\end{align*}
It follows directly from \cite[Lemma 5.3]{vspakula2017coarse} that in the actual median algebra $(\Pi,\mu_\Pi)$, the iterated median $\langle\bar{a}_1,\ldots,\bar{a}_n;\bar{b}\rangle_\Pi$ is nothing but the projection of $\bar{b}$ onto the convex hull of $\bar{a}_1,\ldots,\bar{a}_n$. Hence we have
\[
\langle\bar{a}_1,\ldots,\bar{a}_k;\bar{b}\rangle_\Pi = \langle \bar{a}_1,\ldots,\bar{a}_k;\langle\bar{a}_1,\ldots,\bar{a}_n;\bar{b}\rangle_\Pi \rangle_\Pi.
\]
Combining the above together and taking $E_n:=\rho_n(nh_{n+1}+\rho_n((n+1)h_{n+1}) + H_n(h_{n+1})) + \rho_n((n+1)h_{n+1}) + 2H_n(h_{n+1})$, Equality (\ref{EQ3}) holds.
\end{proof}

\section{Coarse interval structures}\label{coarseintervalstructures}
Sholander studied the relation between intervals and median operators, and we will generalise this approach to the coarse context.

Classically Sholander defined the interval between two points $a$ and $b$ in a median algebra $(X,\mu)$ to be the set $\{c:\mu(a,c,b)=c\}$. This (in the context of median algebras) agrees with our definition of interval (Definition \ref{interval}) since for any $c=\mu(a,x,b)\in [a,b]$, we have
\[
\mu(a,c,b)=\mu(c,a,b)=\mu(\mu(x,a,b),a,b)= \mu(x,\mu(a,b,a),b)=\mu(x,a,b)=c.
\]
Of course the two definitions of interval do not necessarily coincide in the coarse context.

\begin{thm}[Sholander, \cite{sholander1954medians}]\label{sholander}
For every median algebra $(X,\mu)$, the binary operation $[\cdot ,\cdot ]: X\times X\rightarrow \mathcal P(X)$ defined by $(a,b)\mapsto [a,b]$ has the following properties:
\begin{itemize}
\item $[a,a] =\{a\}$,
\item if $c\in [a,b]$ then $[a,c]\subseteq [b,a]$,
\item $[a,b]\cap [b,c]\cap [c,a]$ has cardinality $1$.
\end{itemize}
Conversely, every operation $X^2 \rightarrow \mathcal P(X)$ with the preceding properties induces a ternary operator $\mu'$ whereby $\mu(a,b,c)'$ is the unique point in $[a,b]\cap [b,c]\cap [c,a]$ such that $(X,\mu')$ is a median algebra.
\end{thm}

As remarked by the referee, it requires a little work to extract the proof of the converse statement from Sholander, however we are fortunate that this is explained in some detail in \cite[Section 2]{bowditch2016some}. Here we will provide a coarse analogue of Sholander's theorem. We start by considering the properties of intervals in a coarse median space.

\begin{prop}\label{coarse interval}
Let $(X,d,\mu)$ be a coarse median space with parameters $\rho,\kappaiv$ and $\kappav$. Then the map $[\cdot,\cdot]: X^2 \rightarrow \mathcal{P}(X)$ defined by $(a,b) \mapsto [a,b]=\{\mu(a,x,b)\mid x\in X\}$ satisfies:
\begin{itemize}
  \item[(I1).] For all $a,b\in X$, $[a,a]=\{a\}$, $[a,b]=[b,a]$;
  \item[(I2).] There exists a non-decreasing function $\f: \Rp \rightarrow \Rp$ such that for any $a,b\in X$ and $c\in \mathcal{N}_R([a,b])$, we have $[a,c] \subseteq \mathcal{N}_{\f(R)}([a,b])$;
  \item[(I3).] There exists a non-decreasing function $\g: \Rp \rightarrow \Rp$ such that for any $a,b,c\in X$, we have $[a,b] \cap [b,c] \cap [c,a] \neq \emptyset$ and
$$\diam( \mathcal{N}_R([a,b]) \cap \mathcal{N}_R([b,c]) \cap \mathcal{N}_R([c,a]) ) \leqslant \g(R).$$
\end{itemize}
\end{prop}

\begin{proof}
Property (I1) follows directly from axioms (M1) and (M2) for a coarse median space. Now we consider (I2). Since $c\in\mathcal{N}_R([a,b])$, there exists $x\in X$ such that $c \thicksim_R \mu(a,b,x)$. Then it follows from axioms (C1), (C2) and (M2) that for any $y\in X$ we have
$$\langle a,c,y\rangle \thicksim_{\rho(R)}\langle a,\mu(a,b,x),y\rangle \thicksim_{\kappaiv} \langle a,b,\langle a,x,y\rangle\rangle,$$
which implies $\langle a,c,y\rangle \in \mathcal{N}_{\rho(R)+\kappaiv}([a,b])$. So we can take $\f(R)=\rho(R)+\kappaiv$, and (I2) holds. For (I3), we know that $\mu(a,b,c) \in [a,b] \cap [b,c] \cap [c,a]$ so the intersection is non-empty. Furthermore given a point $z\in \mathcal{N}_R([a,b]) \cap \mathcal{N}_R([b,c]) \cap \mathcal{N}_R([c,a])$ there exists $w\in X$ such that $z \thicksim_R \mu(a,b,w)$. So by (C1) and (C2), we have
$$\mu(a,b,z) \thicksim_{\rho(R)} \mu(a,b,\mu(a,b,w)) \thicksim_\kappaiv \mu(\mu(a,b,a),b,w)=\mu(a,b,w) \thicksim_R z.$$
Similarly for $b,c$ and for $c,a$. Hence we obtain that
$$\mu(a,b,z) \thicksim_{\kappa'} z,\quad \mu(b,c,z) \thicksim_{\kappa'} z,\quad \mu(c,a,z) \thicksim_{\kappa'} z,$$
where $\kappa':=\rho(R)+R+\kappaiv =\f(R)+R$. Combining with (C1) and (\ref{five point estimate}), we obtain
\begin{eqnarray*}
    z & \thicksim_{\kappa'} & \mu(c,a,z) \thicksim_{\rho(\kappa')} \mu(c,a,\mu(b,c,z)) \thicksim_{\kappaiv}  \mu(\mu(c,a,b),c,z)\\
      & = &  \mu(\mu(a,b,c),c,z) \thicksim_{\rho(\kappa')} \mu(\mu(a,b,c),c,\mu(a,b,z)) \thicksim_{\kappav} \mu(a,b,\mu(c,c,z))\\
      & = & \mu(a,b,c).
\end{eqnarray*}
The above estimate implies that the diameter of $\mathcal{N}_R([a,b]) \cap \mathcal{N}_R([b,c]) \cap \mathcal{N}_R([c,a])$ is bounded by
$$\g(R)=4\rho(\kappa')+2\kappa'+4\kappaiv=4\rho(\rho(R)+R+\kappaiv)+2\rho(R)+2R+6\kappaiv.$$
\end{proof}

With this in mind, we define the concept of a coarse interval space as follows.
\begin{defn}
Let $(X,d)$ be a metric space and $[\cdot,\cdot]: X^2 \rightarrow \mathcal{P}(X)$ be a map satisfying (I1)$\sim$(I3) in Proposition \ref{coarse interval}. Then $\I=(X,d,[\cdot,\cdot])$ is called a \emph{coarse interval space}. The functions $\f,\g$ in the conditions are called \emph{parameters} for $\I$. As with the notion of a coarse median space, the parameters are not uniquely defined and are not part of the data. It is only their existence that is required.
\end{defn}

Note that conditions (I1) and (I3) together imply that any interval $[a,b]$ must contain $a$, since the intersection $[a,a]\cap [a,b]\cap [b,a]$ is simultaneously non-empty and contained in $[a,a]:=\{a\}$. Since $[a,b]=[b,a]$ by (I1) as well, it must also contain $b$.

Given a coarse median space $(X,d,\mu)$,  the triple $(X,d,[\cdot, \cdot])$ given by $[a,b]:=\{\mu(a,x,b):x\in X\}$ is said to be the \emph{coarse interval space induced by $(X,d,\mu)$}.

On the other hand, suppose we are given a coarse interval space $(X,d,[\cdot,\cdot])$. Axiom (I3) implies that for any $a,b,c\in X$ we can always choose a point in $[a,b] \cap [b,c] \cap [c,a]$, denoted by $\mu(a,b,c)$, which is invariant under any permutation of $\{a,b,c\}$ (i.e., the choice satisfies axiom (M2)), while (I1) and (I3) together ensure that we can only choose $a$ for the triple $a,a,b$ ensuring that it also satisfies (M1). Making such a choice for all $a,b,c$ gives us a ternary operator $\mu$ on $X$ which we will refer to as the \emph{induced (ternary) operator}.
By axiom (I3), $\mu$ is uniquely determined up to bounded error.

Our proof that the induced ternary operator is a coarse median operator on $X$ is inspired by Sholander's argument in \cite{sholander1954medians}, though more care needs to be taken with the estimates introduced by the coarse conditions. For clarity we divide the proof into several lemmas.

\begin{lem}\label{C1}
Let $(X,d,[\cdot,\cdot])$ be a coarse interval space and $\mu$ be the induced operator. Given parameters $\f,\g$ for the space, then for any $a,a',b,c\in X$, we have
$$d(\mu(a,b,c), \mu(a',b,c)) \leqslant \g(\f(d(a,a'))).$$
In particular, axiom (C1) holds for $(X,d,\mu)$ with $\rho=\g\circ \f$.
\end{lem}

\begin{proof}
Set $R=d(a,a')$, then $a' \in \mathcal{N}_R([a,b])$ and $a' \in \mathcal{N}_R([c,a])$. By (I1) and (I2), we have
$$[a',b] \subseteq \mathcal{N}_{\f(R)}([a,b])\quad\mbox{and}\quad [c, a'] \subseteq \mathcal{N}_{\f(R)}([c,a]).$$
Hence
$$\mu(a',b,c) \in [a',b] \cap [b,c] \cap [c,a'] \subseteq \mathcal{N}_{\f(R)}([a,b]) \cap \mathcal{N}_{\f(R)}([b,c]) \cap \mathcal{N}_{\f(R)}([c,a]).$$
Combined with (I3), we obtain that $\mu(a',b,c) \thicksim_{\g(\f(R))} \mu(a,b,c)$.
\end{proof}

CONVENTION: Following this lemma, given parameters $\f, \g$ we will fix the function $\rho:=3\g\circ \f$ so that $d(\mu(a,b,c), \mu(a',b',c')) \leqslant \rho(d(a,a')+d(b,b')+d(c,c')).$

We now turn our attention to axiom (C2). Fix a coarse interval space $(X,d,[\cdot,\cdot])$ with parameters $\f,\g$ and the induced operator $\mu$. We begin with the following elementary lemma, which can be deduced directly from the definition.
\begin{lem}\label{estm1}
If $c \thicksim_R \mu(a,b,c)$, then $c\in \mathcal{N}_R([a,b])$; conversely, if $c\in \mathcal{N}_R([a,b])$ then $c \thicksim_{\g(R)} \mu(a,b,c)$ for any $a,b,c\in X$.
\end{lem}

The following estimates are a little less obvious.

\begin{lem}\label{estmA}
Let $b\in \mathcal{N}_{R_1}([a,c])$ and $c\in \mathcal{N}_{R_2}([a,d])$. Then $c\in \mathcal{N}_{h(R_1,R_2)}([b,d])$ where $h(R_1,R_2)=\g(R_2)+\g(\f(R_1+\f(R_2)))$.
\end{lem}

\begin{proof}
Since $b\in \mathcal{N}_{R_1}([a,c])$, axioms (I1) and (I2) imply that $[b,c] \subseteq \mathcal{N}_{\f(R_1)}([a,c])$. Since $c\in \mathcal{N}_{R_2}([a,d])$, again by (I2) we have $[a,c] \subseteq \mathcal{N}_{\f(R_2)}([a,d])$. Hence $b \in \mathcal{N}_{R_1}([a,c]) \subseteq \mathcal{N}_{R_1+\f(R_2)}([a,d])$, and consequently $[b,d] \subseteq \mathcal{N}_{\f(R_1+\f(R_2))}([a,d])$ by axioms (I1) and (I2). Combining them together with axiom (I3), we have
$$\mu(b,c,d)\in [b,c]\cap[c,d]\cap[d,b] \subseteq \mathcal{N}_{\f(R_1)}([a,c]) \cap [c,d] \cap \mathcal{N}_{\f(R_1+\f(R_2))}([a,d]),$$
which implies $\mu(b,c,d) \thicksim_{\g(\f(R_1+\f(R_2)))} \mu(a,c,d) \thicksim_{\g(R_2)} c$ (we use Lemma \ref{estm1} in the second estimate since $c\in \mathcal{N}_{R_2}([a,d])$). So the conclusion holds.
\end{proof}

\begin{cor}\label{corA}
Suppose the Hausdorff distance $d_H([a,b],[a,c]) \leqslant R$, then $d(b,c)\leqslant h(R,R)$.
\end{cor}

\begin{proof}
By assumption, $b\in \mathcal{N}_{R}([a,c])$ and $c\in \mathcal{N}_{R}([a,b])$. Now putting $d:=b$ and applying Lemma \ref{estmA}, we have $c\in \mathcal{N}_{h(R,R)}([b,b])$. Since $[b,b]=\{b\}$ by axiom (I1), we have $d(b,c)\leqslant h(R,R)$.
\end{proof}

\begin{lem}\label{estmB}
For any $a,b,c,d\in X$, we have $\mu(a,\mu(a,c,d),\mu(b,c,d)) \thicksim_{\kappa''} \mu(a,c,d)$, where $\kappa''=\g(\f(0)+\g\f^2(0))$. Here we use the notation $\f^2(0):=\f \circ \f(0)$.
\end{lem}

\begin{proof}
Setting $x=\mu(b,c,d)$, we consider $m=\mu(a,\mu(a,x,c),d)\in [a,\mu(a,x,c)] \subseteq \mathcal{N}_{\f(0)}([a,x])$. Taking $y=\mu(a,x,c)=\mu(a,\mu(b,c,d),c)\in [a,c]$, we have $[a,y] \subseteq \mathcal{N}_{\f(0)}([a,c])$ by (I2), which implies $m \in \mathcal{N}_{\f(0)}([a,c])$. Again by (I2), $y \in [c,\mu(b,c,d)] \subseteq \mathcal{N}_{\f(0)}([c,d])$, so $m\in [y,d] \subseteq \mathcal{N}_{\f^2(0)}([c,d])$. Combining them together, we have
$$m\in \mathcal{N}_{\f(0)}([a,c]) \cap \mathcal{N}_{\f^2(0)}([c,d]) \cap [a,d],$$
which implies $\mu(a,c,d)\thicksim_{\g(\f^2(0))}m$ by (I3). Hence $\mu(a,c,d) \in \mathcal N_{\f(0)+\g\f^2(0)}([a,x])$. Finally, by Lemma \ref{estm1}, we have $\mu(a,\mu(a,c,d),x)\thicksim_{\g(\f(0)+\g\f^2(0))} \mu(a,c,d)$.
\end{proof}

From now on, let us fix the constant $\kappa''=\g(\f(0)+\g\f^2(0))$.

\begin{lem}\label{estmC}
For any $R_1,R_2>0$, there exists a constant $\lambda(R_1,R_2)>0$ such that for any $b\in \mathcal{N}_{R_1}([a,c]) \cap \mathcal{N}_{R_2}([a,d])$ and $x\in [c,d]$ we have $b\in \mathcal{N}_{\lambda(R_1,R_2)}([a,x])$. In particular, taking $x=\mu(a,c,d)$ we have:
$$\mathcal{N}_{R_1}([a,c]) \cap \mathcal{N}_{R_2}([a,d]) \subseteq \mathcal{N}_{\lambda(R_1,R_2)}([a,\mu(a,c,d)]).$$
\end{lem}

\begin{proof}
Since $b\in \mathcal{N}_{R_1}([a,c])$, it follows from Lemma \ref{C1} and \ref{estmB} that
$$\mu(d,\mu(a,c,d),b) \thicksim_{\rho(\g(R_1))} \mu(d,\mu(a,c,d),\mu(a,b,c)) \thicksim_{\kappa''} \mu(a,c,d).$$
This implies $\mu(a,c,d)\in \mathcal{N}_{\rho(\g(R_1))+\kappa''}([b,d])$. Together with $b\in \mathcal{N}_{R_2}([a,d])$ and Lemma \ref{estmA}, we have
$b\in \mathcal{N}_{h(\rho(\g(R_1))+\kappa'',R_2)}([a,\mu(a,c,d)])$. On the other hand, since $x\in [c,d]$, it follows from Lemma \ref{C1}, \ref{estm1} and \ref{estmB} that
$$\mu(a,\mu(a,c,d),x) \thicksim_{\rho(\g(0))} \mu(a,\mu(a,c,d),\mu(x,c,d)) \thicksim_{\kappa''} \mu(a,c,d).$$
This implies $\mu(a,c,d)\in \mathcal{N}_{\rho(\g(0))+\kappa''}([a,x])$, hence $[a,\mu(a,c,d)] \subseteq \mathcal{N}_{\f(\rho(\g(0))+\kappa'')}([a,x])$. Combining them together, we have
$$b\in \mathcal{N}_{h(\rho(\g(R_1))+\kappa'',R_2)+\f(\rho(\g(0))+\kappa'')}([a,x]).$$
Now taking
$$\lambda(R_1,R_2)=h(\rho(\g(R_1))+\kappa'',R_2)+\f(\rho(\g(0))+\kappa''),$$
the lemma holds.
\end{proof}

Finally we are in the position to prove the following theorem.
\begin{thm}\label{coarse interval converse}
Let $(X,d,[\cdot,\cdot])$ be a coarse interval space with the induced operator $\mu$, then $(X,d,\mu)$ is a coarse median space.
\end{thm}

\begin{proof}
It only remains to verify (C2). In other words, we need to find a constant $\kappa$ such that for any $a,b,c,d\in X$, we have
$$\mu( \mu(a,b,c),b,d ) \thicksim_\kappa \mu( a,b,\mu(c,b,d) ).$$
By axiom (I2) and Lemma \ref{estmC}, we have:
\begin{eqnarray*}
  [b,\mu( \mu(a,b,c),b,d )] &\subseteq& \mathcal{N}_{\f(0)}([b,\mu(a,b,c)]) \cap \mathcal{N}_{\f(0)}([b,d]) \\
  &\subseteq& \mathcal{N}_{\f^2(0)}([b,a]) \cap \mathcal{N}_{\f^2(0)}([b,c]) \cap \mathcal{N}_{\f(0)}([b,d])\\
  &\subseteq& \mathcal{N}_{\f^2(0)}([b,a]) \cap \mathcal{N}_{\lambda(\f^2(0),\f(0))}([b,\mu(b,c,d)])\\
  &\subseteq& \mathcal{N}_{\lambda(\f^2(0),\lambda(\f^2(0),\f(0)))}([b,\mu(a,b,\mu(b,c,d))]).
\end{eqnarray*}
Similarly we have
$$[b,\mu(a,b,\mu(b,c,d))] \subseteq \mathcal{N}_{\lambda(\f^2(0),\lambda(\f^2(0),\f(0)))}([b,\mu( \mu(a,b,c),b,d )]).$$
The above two estimates imply:
$$d_H([b,\mu( \mu(a,b,c),b,d )], [b, \mu(a,b,\mu(c,b,d ))]) \leqslant \lambda(\f^2(0),\lambda(\f^2(0),\f(0))).$$
Finally it follows from Corollary \ref{corA} that
$$\mu( \mu(a,b,c),b,d ) \thicksim_\kappa \mu(a,b,\mu(c,b,d ))$$
for $\kappa=h(\lambda(\f^2(0),\lambda(\f^2(0),\f(0))),\lambda(\f^2(0),\lambda(\f^2(0),\f(0))))$.
\end{proof}

Analogous to relaxing axioms (M1) and (M2) for a coarse median operator to axiom (C0), we consider the following notion of a coarse interval structure.
\begin{defn}\label{def: coarse interval structure}
A \emph{coarse interval structure} on a metric space $(X,d)$ is a triple $\I=(X,d,[\cdot,\cdot])$, where $[\cdot,\cdot]$ is a map from $X^2$ to $\mathcal{P}(X)$ such that  there exists a constant $\kappao >0$ for which  the following conditions hold:
\begin{itemize}
  \item[(I1)'.] For all $a,b\in X$, $d_H([a,a], \{a\}) \leqslant \kappao$ and $d_H([a,b],[b,a])\leqslant \kappao$;
  \item[(I2).] There exists a non-decreasing function $\f: \Rp \rightarrow \Rp$ such that for any $a,b\in X$ and $c\in \mathcal{N}_R([a,b])$, we have $[a,c] \subseteq \mathcal{N}_{\f(R)}([a,b])$;
  \item[(I3)'.] There exists a non-decreasing function $\g: [\kappao,+\infty) \rightarrow \Rp$ such that for any $a,b,c\in X$ and $R \geqslant \kappao$, we have $\N_\kappao([a,b]) \cap \N_\kappao([b,c]) \cap \N_\kappao([c,a]) \neq \emptyset$ and $\diam( \mathcal{N}_R([a,b]) \cap \mathcal{N}_R([b,c]) \cap \mathcal{N}_R([c,a]) ) \leqslant \g(R)$.
\end{itemize}
The constant $\kappao$ and functions $\f,\g$ in the conditions are called \emph{parameters} for $\I$.
\end{defn}

\begin{rem}\label{ends close to intervals}
By (I1)', for any point $a$ the interval $[a,a]$ lies in $B(a,\kappao)$. By (I3)', the intersection $\N_\kappao([a,a]) \cap \N_\kappao([a,b])$ must be non-empty for all $b$. As $\N_\kappao([a,a])$ lies in $B(a,2\kappao)$, it follows that $a$ must lie in $\N_{3\kappao}([a,b])$. Similarly $b\in\N_{3\kappao}([a,b])$.
\end{rem}

Recall that every coarse median is  uniformly close to some coarse median satisfying axioms (M1) and (M2). Similarly we will show that a coarse interval structure is always ``close" to another satisfying (I1)$\sim$(I3) in the following sense.
\begin{defn}
Let $(X,d, [\cdot, \cdot])$ and $(X,d, [\cdot, \cdot]')$ be coarse interval structures. We say they are \emph{uniformly close}  if there exists a constant $C>0$ such that $d_H([x,y],[x,y]')\leqslant C$ for any $x,y\in X$.
\end{defn}

\begin{lem}\label{coarse interval close}
Let $(X,d,[\cdot,\cdot])$ be a coarse interval structure. Then there exists a map $[\cdot, \cdot]'$ which is uniformly close to $[\cdot, \cdot]$ and such that $(X,d,[\cdot, \cdot]')$ is a coarse interval space.
\end{lem}

\begin{proof}
We define `fattened' intervals:
\[
[a,b]':= \N_\kappao([a,b])\cup \N_\kappao([b,a])\cup \{a,b\}
\]
for $a\neq b$, and define $[a,a]':=\{a\}$.
It is easy to see from (I1)' that $[a,a]'=\{a\}$ is $\kappa_0$-close to $[a,a]$ and that $\N_\kappao([a,b])\cup \N_\kappao([b,a])$ is $2\kappa_0$-close to $[a,b]$. By Remark \ref{ends close to intervals}, the points $a,b$ are $3\kappa_0$-close to $[a,b]$, hence $[a,b]'$ is $3\kappa_0$-close to $[a,b]$.

By construction, $[\cdot,\cdot]'$ satisfies (I1) and clearly it still satisfies (I2). The fattening of the intervals together with (I3)' ensures that $[a,b]'\cap[b,c]'\cap[c,a]'$ is non-empty for $a,b,c$ distinct. Now taking repeated points, $[a,b]'\cap[b,b]'\cap[b,a]'=\{b\}$ by construction. Hence $[a,b]'\cap[b,c]'\cap[c,a]'$ is non-empty in all cases. The above analysis shows that the $R$-neighbourhood of the interval $[a,b]'$ is contained in the $(R+3\kappa_0)$-neighbourhood of the interval $[a,b]$, so the intersection $ \mathcal{N}_R([a,b]') \cap \mathcal{N}_R([b,c]') \cap \mathcal{N}_R([c,a]')$ is contained in the intersection $ \mathcal{N}_{R+3\kappa_0}([a,b]) \cap \mathcal{N}_{R+3\kappa_0}([b,c]) \cap \mathcal{N}_{R+3\kappa_0}([c,a])$. It therefore has diameter bounded by  $\psi(R+3\kappa_0)$ by (I3)'.  This establishes (I3).
\end{proof}

Adapting the arguments we made above, we have the following correspondence between coarse median structures and coarse interval structures.
\begin{thm}\label{induce cm and ci}
\begin{enumerate}
  \item Given a coarse median structure $(X,d,\mu)$, the map $(a,b)\mapsto [a,b]$ provided by Definition \ref{interval} gives  an \emph{induced coarse interval structure} $(X,d,[\cdot,\cdot])$.
  \item Let $(X,d,[\cdot,\cdot])$ be a coarse interval structure with parameters $\kappao,\f,\g$. For any $a,b,c\in X$, choose a point in $\N_\kappao([a,b]) \cap \N_\kappao([b,c]) \cap \N_\kappao([c,a])$, denoted by $\mu(a,b,c)$. Making such a choice gives an \emph{induced coarse median structure} $(X,d,\mu)$.
  \item Furthermore the above two procedures are inverse to each other up to uniform bounds:
  \begin{itemize}
  \item For a coarse median structure $(X,d,\mu)$ with  induced coarse interval structure $(X,d,[\cdot,\cdot])$, any induced coarse median structure  $(X,d,\mu')$ is uniformly close to $(X,d,\mu)$;
  \item For a coarse interval structure $(X,d, [\cdot, \cdot])$ with any induced coarse median structure $(X,d,\mu)$, the induced coarse interval structure is uniformly close to $(X,d, [\cdot, \cdot])$.
  \end{itemize}
\end{enumerate}
\end{thm}

\begin{proof}
(1). By Remark \ref{unifclosestructure}, $(X,d,\mu)$ is uniformly close to a coarse median space $(X,d,\mu')$. Now Proposition \ref{coarse interval} shows that the map $X^2 \rightarrow \mathcal{P}(X)$ given by $[a,b]'=\{\mu(a,c,b)': c\in X\}$ gives a coarse interval space $(X,d,[\cdot,\cdot]')$. Since $\mu$ and $\mu'$ are uniformly close, $[\cdot, \cdot]$ and $[\cdot,\cdot]'$ are uniformly close as well. Hence $(X,d,[\cdot,\cdot])$ is a coarse interval structure.

(2). From Lemma \ref{coarse interval close}, $(X,d,[\cdot,\cdot])$ is uniformly close to some coarse interval space $(X,d,[\cdot,\cdot]')$. We now apply Theorem \ref{coarse interval converse} to construct an induced coarse median space  $(X,d,\mu')$. Since $[\cdot,\cdot]$ and $[\cdot,\cdot]'$ are uniformly close, $\mu$ and $\mu'$ are uniformly close as well. Hence  $(X,d,\mu)$ is a coarse median structure.

(3). First we start with the coarse median structure $(X,d,\mu)$ with parameters $\rho, \kappao, \kappaiv,\kappav$ and induced coarse interval structure $(X,d,[\cdot,\cdot])$ with some parameters $\kappao',\f,\g$. Let $(X,d,\mu')$ be any induced coarse median structure of $(X,d,[\cdot,\cdot])$. By definition for any $x,y,z\in X$, $\mu(x,y,z)'$ is some point chosen from the intersection $\N_{\kappao'}([x,y])\cap \N_{\kappao'}([y,z])\cap \N_{\kappao'}([z,x])$, \emph{a fortiori} is in the intersection  $\N_{\kappao''}([x,y])\cap \N_{\kappao''}([y,z])\cap \N_{\kappao''}([z,x])$ where $\kappao''=\max\{\kappao,\kappao'\}$. The latter contains $\mu(x,y,z)$ and is uniformly bounded with diameter at most $\psi(\kappao'')$. Hence $\mu$ and $\mu'$ are uniformly close.

Conversely given a coarse interval structure $(X,d,[\cdot,\cdot])$ with parameters $\kappao,\f,\g$ and any induced coarse median structure $(X,d,\mu)$, we consider the induced coarse interval structure $(X,d,[\cdot,\cdot]')$ of $(X,d,\mu)$. By definition for any $x,y,z\in X$, $\mu(x,z,y)$ is some point chosen from $\N_{\kappao}([x,z])\cap \N_{\kappao}([z,y])\cap \N_{\kappao}([y,x])$.  Hence $\mu(x,z,y) \in \N_{\kappao}([y,x])\subseteq \N_{2\kappao}([x,y])$, which implies $[x,y]' \subseteq \N_{2\kappao}([x,y])$. On the other hand for any $z\in [x,y]$, Remark \ref{ends close to intervals} implies $z\in \N_{\kappao}([y,x])\cap \N_{3\kappao}([z,y]) \cap \N_{3\kappao}([x,z])$. It follows that both $z$ and $\mu(x,z,y)$ lie in $\N_{\kappao}([y,x])\cap \N_{3\kappao}([z,y]) \cap \N_{3\kappao}([x,z])$. So by axiom (I3)', we have $z\thicksim_K \mu(x,z,y) \in [x,y]'$ for $K=\psi(3\kappao)>0$. Hence $[x,y]\subseteq \N_K([x,y]')$, which implies $d_H([x,y],[x,y]') \leqslant \max\{2\kappao,K\}$ for any $x,y\in X$. Therefore $[\cdot, \cdot]$ and $[\cdot, \cdot]'$ are uniformly close.
\end{proof}

\section{Rank,  generalised hyperbolicity and interval growth}\label{growth section}

\subsection{Generalised hyperbolicity for higher rank coarse median spaces}\label{generalised hyperbolicity}

Here we will provide the following characterisations of rank for a coarse median space.

\begin{thm}\label{hyper rank}
\hyperrankthm
\end{thm}

As Bowditch showed in \cite{bowditch2013coarse}, a geodesic coarse median space has rank $1$ \emph{if and only if} it is hyperbolic, and it is instructive to consider conditions 2) and 3) above in that context. Here condition 2) reduces to a version of  the generalised slim triangles condition abstracted from classical hyperbolic geometry, while condition 3) reduces to the Gromov inequality (see Equation (\ref{gromov product generalised}) below) motivated by the geometry of trees. From this perspective, Theorem \ref{hyper rank} provides higher rank analogues of these two characterisations.

To be more precise, recall that in \cite[Theorem 4.4]{niblo2017four} we established the following result in the special case where $b$ is required to range over those points in the interval $[a,c]$. To deduce the more general result stated here, one only needs to consider the effect of replacing the general point $b$ by the coarse median $\mu(a,b,c)$ and then use the fact that the intervals $[a, \mu(a,b,c)]$ and $[\mu(a,b,c),c]$ both lie in a uniformly bounded neighbourhood of the interval $[a,c]$:
\begin{thm}[\cite{niblo2017four}]\label{hyperbolic char}
For a coarse median space $(X,d,\mu)$, the following are equivalent:
\begin{enumerate}[1)]
  \item $\rank X \leqslant 1$;
  \item There exists a constant $\delta>0$ such that for any $a,b,c \in X$, we have
  $$[a,c] \subseteq \N_\delta([a,b]) \cup \N_\delta([b,c]).$$
\end{enumerate}
\end{thm}
We also showed in \cite[Theorem 4.2]{niblo2017four} that the intervals in a rank 1 geodesic coarse median space are uniformly close to geodesics, so Theorem \ref{hyperbolic char} is a version of the slim triangles condition for hyperbolicity. Clearly Theorem \ref{hyper rank} generalises this, providing a higher rank analogue of the slim triangles condition which holds even in the non-geodesic context.

\begin{rem}
  The closeness of geodesics and intervals is a unique (and not \emph{a priori} obvious) feature of the rank 1 case. Combining this fact with Proposition \ref{coarse interval}, we deduce that any geodesic metric space admits at most one coarse median of rank one up to uniform bound. As Zeidler showed in \cite[Example 2.2.8]{zeidler2013coarse}, this is \emph{not} true for higher rank cases (indeed, even in rank 2). The classical median on the Euclidean plane (corresponding to the Cartesian coordinates) is given by taking the interval from the origin to the point $(x,y)$ to be the rectilinear area with diagonal between these points and extending this by the translation action to an interval structure on the plane, equipped with its usual metric. Rotating the frame by an angle of $\pi/4$ one obtains a new interval structure but of course the metric does not change. To see that the new median is not equivalent to the standard one, one observes that while the points $(n,0), (0,0), (0,n)$ have standard median $(0,0)$ their median in the new structure is $(\frac{n}{\sqrt{2}}, \frac{n}{\sqrt{2}})$ which diverges to infinity.
\end{rem}

Turning now to Gromov's inner product, we recall the definition. Fixing a base point $p$ in a metric space $(X,d)$ and for $a,b\in X$, we set
$$(a|b)_p:=\frac{1}{2}[d(a,p)+d(b,p)-d(a,b)].$$

\begin{thm}[Gromov, \cite{gromov1987hyperbolic}]
A geodesic metric space $(X,d)$ is Gromov hyperbolic \emph{if and only if} there exists some constant $\delta>0$ such that the following inequality holds for any $a,b,c,p \in X$:
\begin{equation}\label{gromov inequality}
  \min\{(a|b)_p,(b|c)_p\} \leqslant (a|c)_p + \delta.
\end{equation}
\end{thm}

Note that the Gromov product is determined by the properties $(z|y)_x+(z|x)_y=d(x,y)$ and symmetry: $(z|y)_x=(y|z)_x$ for all $x,y,z$.

If $i_y$ is the intermediate point on a geodesic from $x$ to $z$ such that $d(x,i_y)=(y|z)_x$ and $d(z,i_y)=d(x,z)-d(x,i_y)=(y|x)_z$, then we have
\begin{align*}
  (y|x)_{i_y}&=1/2(d(y,i_y)+d(x,i_y)-d(y,x))\\
  &=1/2(d(y,i_y)+(z|y)_x-d(y,x))\\
  &=1/2(d(y,i_y)-(z|x)_y).
\end{align*}
Since this is symmetric in $x,z$ for such points, we obtain $(y|x)_{i_y}=(y|z)_{i_y}$.

\begin{lem}\label{4 point special case}
  A geodesic space $X$ is hyperbolic \emph{if and only if} there exists $\delta>0$ such that for all $x,y,z\in X$ and $p$ on a geodesic from $x$ to $z$, we have
  $$\min\{(y|x)_p,(y|z)_p\}\leq \delta.$$
\end{lem}

\begin{proof}
  The condition is a special case of Gromov's 4-point condition, so is implied by hyperbolicity.

  Now we consider the converse. For any $x,y,z\in X$, let $i_x,i_y,i_z$ be intermediate points on geodesics from $y$ to $z$, $x$ to $z$ and $x$ to $y$ respectively such that $d(x,i_y)=d(x,i_z)=(y|z)_x$, \emph{etc}.

  As noted above $(y|x)_{i_y}=(y|z)_{i_y}$, so by hypothesis both of these are at most $\delta$. Thus we have
  $$d(x,i_x)+d(i_x,y)\leq d(x,y)+2\delta,$$
  meaning intuitively that $i_x$ is `almost' on a geodesic from $x$ to $y$.

  On the other hand, $d(x,i_z)+d(i_z,y)=d(x,y)$ and $d(i_z,y)=(x|z)_y=d(i_x,y)$. So we have $d(x,i_x)\leq d(x,i_z)+2\delta$. We will show a fellow travelling result that bounds $d(i_x,i_z)$.

  Since $i_z$ is on the geodesic from $x$ to $y$, our hypothesis implies one of $(i_x|x)_{i_z},(i_x|y)_{i_z}$ is at most $\delta$. In other words, for some $u\in \{x,y\}$ we have
  $$d(i_x,i_z)+d(u,i_z)\leq d(i_x,u)+2\delta.$$
When $u=y$ we have $d(i_x,u)=d(u,i_z)$, while $u=x$ implies $d(i_x,u)\leq d(u,i_z)+2\delta$. Hence in either case, we have $d(i_x,i_y)\leq 4\delta$.

Interchanging the roles of $x,y,z$ we see that $\{i_x,i_y,i_z\}$ has diameter at most $4\delta$, which implies hyperbolicity by \cite[III.H.1.17(3)]{BH99}.
\end{proof}

We note that neither the Gromov 4-point condition nor our special case of this is quasi-isometry invariant, however we can give a quasi-isometrically invariant form in terms of medians as follows.  To allow for this coarsening we consider the following condition: there exists a non-decreasing function $\varphi$ such that for any $x,y,z;p\in X$, we have
\begin{equation}\label{gromov product generalised}
  \min\{(y|x)_p,(y|z)_p\} \leqslant \varphi((x|z)_p).
\end{equation}
This is \emph{a priori} weaker than the Gromov 4-point condition, but stronger than the hypothesis in Lemma \ref{4 point special case} (taking $\delta=\varphi(0)$), so in the geodesic case Inequality (\ref{gromov product generalised}) characterises hyperbolicity.

Now for a rank $1$ geodesic coarse median space $(X,d,\mu)$, there exists a constant $C>0$ such that for any $a,b,p \in X$, $(a|b)_p \thicksim_C d(p,\langle a,b,p\rangle)$. This follows directly from the fact that intervals and geodesics are uniformly close to each other \cite[Theorem 4.2]{niblo2017four}. Hence in this situation, the coarse inequality (\ref{gromov product generalised}) above can be rewritten to give the following characterisation of hyperbolicity:
\begin{equation}\label{generalised gromov inequality}
  \min\{d(p,\mu(a,b,p)),d(p,\mu(b,c,p))\} \leqslant \varphi(d(p,\mu(a,c,p))).
\end{equation}
This inequality has the virtue that it is quasi-isometry invariant. Hence the class of quasi-geodesic coarse median spaces satisfying this variant of the 4-point condition is closed under quasi-isometry, providing a natural generalisation of geodesic hyperbolic spaces.

Equation (\ref{generalised gromov inequality}) is the rank 1 case of Theorem \ref{hyper rank} (3), so this theorem  provides a higher rank generalisation of the Gromov inner product characterisation of hyperbolicity.
We now turn to the proof of the theorem.

\begin{proof}[Proof of Theorem \ref{hyper rank}]
Assume $(\rho,\kappaiv\jedit{, \kappav})$ are parameters of $(X,d,\mu)$.

\emph{$3) \Rightarrow 2)$}: For any $p \in \bigcap_{i \neq j} \N_\lambda([x_i,x_j])$ and $i \neq j$, there exists $p' \in [x_i,x_j]$ such that $p\thicksim_\lambda p'$. So we have
$$\langle x_i,p,x_j\rangle \thicksim_{\rho(\lambda)} \langle x_i,p',x_j\rangle \thicksim_{\kappaiv} p' \thicksim_\lambda p.$$
Hence from condition (3), there exists some $i=1,\ldots,n+1$ such that
$$d(p,\mu(x_i,p,q))\leqslant \varphi(\rho(\lambda)+\lambda+\kappaiv).$$
Taking $\psi(\lambda)=\varphi(\rho(\lambda)+\lambda+\kappaiv)$, we have $p \in \N_{\psi(\lambda)}([x_i,q])$ as required.

\emph{$2) \Rightarrow 3)$}: For any $p,q$ and $x_1,\ldots,x_{n+1} \in X$, take $\xi=\max\{d(p,\langle x_i,x_j,p\rangle): i\neq j\}$. Then $p\thicksim_{\xi}\langle x_i,x_j,p\rangle \in [x_i,x_j]$. By condition 2), there exists some $i=1,\ldots,n+1$ such that $p \in \N_{\psi(\xi)}([x_i,q])$, i.e., there exists some $p' \in [x_i,q]$ such that $p \thicksim_{\psi(\xi)}p'$. Hence
$$\mu(x_i,p,q) \thicksim_{\rho(\psi(\xi))} \mu(x_i,p',q) \thicksim_{\kappaiv}p' \thicksim_{\psi(\xi)} p.$$
Taking $\varphi(\xi)=\rho(\psi(\xi))+\psi(\xi)+\kappaiv$, we are done.

\emph{1) $\Rightarrow$ 3)}: Since the rank is at most $n$, Theorem \ref{char for high rank-final} implies that for any $\lambda>0$ there exists a constant $C=C(\lambda)$ such that for any $a,b\in X$ and $e_1,\ldots,e_{n+1}\in [a,b]$ with $\langle e_i,a,e_j\rangle\thicksim_\lambda a$ ($i\neq j$), one of $e_i$'s is $C$-close to $a$. Set $\xi=\max\{d(p,\langle x_i,x_j,p\rangle): i\neq j\}$, then by the coarse 4-point axiom (C2) we have:
$$\langle\langle x_i,p,q\rangle,p,\langle x_j,p,q\rangle\rangle \jedit{\thicksim_\kappav} \langle\langle x_i,x_j,p\rangle,p,q \rangle \thicksim_{\rho(\xi)}\langle p,p,q\rangle=p$$
for any $i \neq j$. Therefore we have
$$\min\{d(p,\mu(x_i,p,q)):i=1,\ldots,n+1\}\leqslant C(\rho(\xi)+\kappav).$$
Taking $\varphi(\xi)=C(\rho(\xi)+\kappav)$, we are done.

\emph{3) $\Rightarrow$ 1)}: Assume $e_1,\ldots,e_{n+1} \in [a,b]$ with $\langle e_i,a,e_j\rangle \thicksim_\lambda a$. Condition 3) implies that
$$\min\{d(a,\mu(e_i,a,b)): i=1,\ldots,n+1\} \leqslant \varphi(\lambda).$$
Since $e_i\in [a,b]$, we have $\mu(e_i,a,b)\thicksim_{\kappaiv} e_i$. Hence
$$\min\{d(a,e_i): i=1,\ldots,n+1\} \leqslant \varphi(\lambda)+\kappaiv.$$
Taking $C(\lambda)=\varphi(\lambda)+\kappaiv$, $(X,d,\mu)$ has rank at most $n$ by Theorem \ref{char for high rank-final}.
\end{proof}

This suggests a natural notion of rank for coarse interval spaces as follows.

\begin{defn}
Let $(X,d,[\cdot,\cdot])$ be a coarse interval structure. We say that \emph{the rank of $(X,d,[\cdot,\cdot])$ is at most $n$} if there exists a non-decreasing function $\psi$ such that
$$\bigcap_{i \neq j} \N_\lambda([x_i,x_j]) \subseteq \bigcup_{i=1}^{n+1} \N_{\psi(\lambda)}([x_i,q])$$
for any $\lambda>0$ and $x_1,\ldots,x_{n+1},q \in X$.
\end{defn}

Note that in the higher rank case ($n\geq 2$), the intersection on the left must be uniformly bounded by axiom (I3) and can be thought of as a generalised centroid of the points $x_1, \ldots, x_{n+1}$.
So the axiom asserts that the generalised centroid must be close to at least one of those coarse intervals.

With this definition and combining Theorem \ref{induce cm and ci}, we obtain the following:

\begin{cor}\label{rank preserving}
For a metric space, any coarse median of rank $n$ induces a coarse interval structure of rank $n$ and vice versa.
\end{cor}

\subsection{Cubes in coarse median spaces}\label{cubesinCMA}
In this subsection we will provide a structure theorem which describes a coarse cube in a coarse median space as a product of coarse intervals. It will play a key role in our characterisation of finite rank coarse median spaces in terms of the growth of coarse intervals.

Recall that median cubes are the fundamental building blocks for median algebras. Equipping the median $n$-cube $(I^n,\mu_n)$ with the $\ell^1$-metric $d_{\ell^1}$ makes it a coarse median space $(I^n, d_{\ell^1}, \mu_n)$.

\begin{defn}\label{coarsecubedef}
An \emph{$L$-coarse cube} of rank $n$ in a coarse median structure $(X,d,\mu)$ is an $L$-quasi-morphism $c$ from $(I^n,\mu_n)$ to $(X,d,\mu)$. An \emph{edge} in an $L$-coarse cube $c$ is a pair of points $c(\bar a), c(\bar b)$ in the image such that $\bar a, \bar b$ are adjacent vertices in the median cube. Two edges in an $L$-coarse cube $c$ are said to be \emph{parallel} if there exist parallel edges in the median cube which map to them under $c$.
\end{defn}

We will denote the origin of the median $n$-cube by $\bar{\0}$, the vertex diagonally opposite to $\bar{\0}$ by $\bar{\1}$ and the vertices adjacent to $\bar{\0}$ by $\bar{e}_1,\ldots,\bar{e}_n$. Given an $L$-coarse cube $c$, where there is no risk of confusion we will denote the images of the vertices $\bar{\0},\bar{\1},\bar{e}_1,\ldots,\bar{e}_n$ under the map $c$ by $\0, \1, e_1, \ldots, e_n$ respectively. The convention that elements of the median cube are barred while their images are not corresponds to the view that the median cube is an approximation (in the sense of Bowditch, see Definition \ref{Bowditch original def}) to the finite set of vertices $\0, \1, e_1, \ldots, e_n$.

Note that in Definition \ref{coarsecubedef} we do not impose any control on the distances between the points of the image, since we wish to allow cubes of arbitrarily large diameter. By analogy with Zeidler's result in \cite{zeidler2013coarse}, we have the following lemma, which controls the relationship between lengths of parallel edges in a coarse cube.

\begin{lem}
Given an edge $e$ of length $d$ in an $L$-coarse cube $c$, all edges parallel to $e$ in $c$ have length bounded by $\rho(d)+2L$, where $\rho$ is a control function parameter for the coarse median.
\end{lem}

The proof is similar to that of \cite[Lemma 2.4.5]{zeidler2013coarse} and is therefore omitted. Given that there is control between the lengths of parallel edges but no control on the lengths of ``perpendicular'' edges, it may be helpful to think of a coarse cube as a coarse cuboid.

\begin{defn}
Given an interval $[a,b]$ in a coarse median structure $(X, d,\mu)$, we may define a new ternary operator on $[a,b]$ by $\langle x,y,z \rangle_{a,b}:= \langle a,\langle x,y,z\rangle,b\rangle$. By \cite[Lemma 2.22]{niblo2017four}, the triple $([a,b], d|_{[a,b]}, \mu_{a,b})$ is a coarse median structure and $\mu\thicksim_C \mu_{a,b}$, where $C$ is independent of $a,b$.
\end{defn}

Given an $L$-coarse cube $f:I^n\rightarrow X$, define the following coarse median spaces:
\[
\mathcal A:=([\0,\1], d, \mu_{\0,\1}); \quad\mathcal B:=([\0,e_1]\times \ldots \times [\0, e_n], d_{\ell^1}, \mu_{\ell^1})
\]
where $d_{\ell^1}$ denotes the $\ell^1$-product of the induced metrics on the intervals $[\0, e_i]$ and $\mu_{\ell^1}$ is defined by $\mu_{\ell^1}=\mu_{\0,e_1}\times \ldots \times \mu_{\0, e_n}$. Also define maps as follows:
$$
\left.
   \begin{array}{ccccc}
      \Phi:\mathcal A\rightarrow \mathcal B, &\quad& x & \mapsto & (\mu(\0,x,e_1), \ldots, \mu(\0,x,e_n));\\
      \Psi:\mathcal B\rightarrow \mathcal A, &\quad&(x_1, \ldots, x_n)& \mapsto & \mu(\mu(x_1,\ldots, x_n;\1),\0,\1).
   \end{array}
\right.
$$

\begin{thm}\label{productresult}
Let $(X,d,\mu)$ be a coarse median space and $f:I^n\rightarrow X$ be an $L$-coarse cube of rank $n$ in $X$. Then the map $\Phi:\mathcal A \to\mathcal B$ defined above provides a $(\rho_+,C)$-coarse median isomorphism with inverse $\Psi$ defined above, where $\rho_+(t)=Kt+H_0$ and $K,H_0,C$ depend only on $n$, $L$ and parameters of $(X,d,\mu)$.
\end{thm}

\begin{proof}
Assume $\rho,\kappaiv,\kappav$ are parameters of $(X,d,\mu)$. First we show that $\Phi,\Psi$ are bornologous. By axiom (C1), for any $x,y \in [\0,\1]$ we have:
\begin{eqnarray*}
  d_{\ell^1}(\Phi(x),\Phi(y)) = \sum_{k=1}^n d(\mu(\0,x,e_k),\langle\0,y,e_k\rangle) \leqslant \sum_{k=1}^n \rho(d(x,y))
   = n\rho(d(x,y)),
\end{eqnarray*}
which implies $\Phi$ is $(n\rho)$-bornologous. On the other hand, for any $\vec{x}=(x_1, \ldots, x_n)$ and $\vec{y}=(y_1, \ldots, y_n) \in [\0,e_1]\times \ldots \times [\0, e_n]$, axiom (C1) implies:
\begin{eqnarray*}
  d(\Psi(\vec{x}),\Psi(\vec{y})) &=& d( \mu(\mu(x_1,\ldots, x_n;\1),\0,\1), \langle\langle y_1,\ldots, y_n;\1\rangle,\0,\1\rangle ) \\
   &\leqslant& \rho(d(\mu(x_1,\ldots, x_n;\1),\langle y_1,\ldots, y_n;\1\rangle))\\
   &\leqslant& \rho\circ \rho_n(\sum_{k=1}^n d(x_k,y_k)).
\end{eqnarray*}
Here the last inequality follows from the control over iterated coarse medians provided by Lemma \ref{coarse iterated estimate}(1). This implies $\Psi$ is $(\rho\circ \rho_n)$-bornologous. Since $\rho$ and $\rho_n$ are both affine so is the function $\rho_+:= n\rho+\rho\circ\rho_n$.

Next we show that $\Phi$ is a quasi-morphism. For $x,y,z \in [\0,\1]$, $\mu(x,\0,\1)\thicksim_\kappaiv x$ and $\mu(y,\0,\1)\thicksim_\kappaiv y$. So by axiom (C1) and the estimate (\ref{five point estimate}), we have
$$\langle\langle x,y,z\rangle,\0,\1\rangle \thicksim_{\kappav}\langle\langle x,\0,\1\rangle,\langle y,\0,\1\rangle,z\rangle \thicksim_{\rho(2\kappaiv)}\langle x,y,z\rangle.$$
Applying the same argument again, denoting the projection from $[\0,e_1]\times \ldots \times [\0, e_n]$ onto the $i$-th coordinate by $pr_i$, we have:
\begin{eqnarray*}
&&pr_i\circ \Phi (\langle x,y,z \rangle_{\0,\1}) = \langle\0, \langle\langle x,y,z\rangle,\0,\1\rangle, e_i \rangle \thicksim_{\rho(\rho(2\kappaiv)+\kappav)} \langle \0,\langle x,y,z\rangle,e_i\rangle \\
&\thicksim_\kappaiv& \langle \0, \langle\0,\langle x,y,z\rangle,e_i\rangle, e_i \rangle \thicksim_{\rho(\kappav)} \langle \0, \langle\langle\0,x,e_i\rangle,\langle\0,y,e_i\rangle,z\rangle, e_i \rangle  \\
&\thicksim_\kappav& \langle\langle\0,x,e_i\rangle, \langle\0,\langle\0,y,e_i\rangle,e_i\rangle, \langle\0,z,e_i\rangle\rangle
\thicksim_{\rho(\kappaiv)} \langle\langle\0,\langle\0,x,e_i\rangle,e_i\rangle, \langle\0,\langle\0,y,e_i\rangle,e_i\rangle, \langle\0,z,e_i\rangle\rangle\\
&\thicksim_{\kappav}& \langle\0, \langle\langle \0,x,e_i\rangle, \langle\0,y,e_i\rangle, \langle\0,z,e_i\rangle \rangle, e_i \rangle
=pr_i(\langle\Phi(x),\Phi(y),\Phi(z)\rangle_{\ell^1}).
\end{eqnarray*}
Hence $\Phi$ is a $C'$-quasi-morphism for $C'=n[\rho(\rho(2\kappaiv)+\kappav)+\rho(\kappaiv)+\rho(\kappav)+\kappaiv+2\kappav]$.

Note that in the canonical cube $I^n$, the iterated median $\langle\bar{e}_1, \ldots, \bar{e}_n;\bar{\1}\rangle_n=\bar{\1}$. It follows that by Lemma \ref{coarse iterated estimate}(2), there exists a constant $H_n(L)$ such that
$$\mu(e_1, \ldots,e_n;\1)=\langle f(\bar{e}_1), \ldots, f(\bar{e}_n);f(\bar{\1})\rangle\thicksim_{H_n(L)}f(\langle\bar{e}_1, \ldots, \bar{e}_n;\bar{\1}\rangle_n)=f(\bar{\1})=\1.$$
Now by Lemma \ref{coarse iterated estimate}(3), there is a constant $C_n$ such that for any $x \in [\0,\1]$ we have
\begin{eqnarray*}
  \Psi \circ \Phi(x) &=& \mu(\mu(\mu(\0,x,e_1), \ldots, \mu(\0,x,e_n);\1),\0,\1)\\
   &\thicksim_{\rho(C_n)}&\mu(\mu(\0,x,\mu(e_1,\ldots, e_n;\1)),\0,\1) \thicksim_{\rho^2(H_n(L))}\mu(\mu(\0,x,\1),\0,\1)\\
&\thicksim_{\kappaiv}& \mu(x,\0,\1)\thicksim_{\kappaiv} x.
\end{eqnarray*}
Hence $ \Psi \circ \Phi$ is $C''$-close to the identity on $\mathcal A$ for $C'':=\rho^2(H_n(L))+\rho(C_n)+2\kappaiv$.

Since $f$ is an $L$-coarse median morphism, we have $\mu(\0,\1,e_i)\thicksim_{L} e_i$ and $\langle\0,e_i, e_j\rangle\thicksim_L \0$ for $i\neq j$. For any $\vec{x}=(x_1, \ldots, x_n) \in [\0,e_1]\times \ldots \times [\0, e_n]$, we have:
\begin{eqnarray*}
pr_i \circ \Phi \circ \Psi (\vec{x}) &=& \mu(\0,\mu(\mu(x_1,\ldots,x_n;\1),\0,\1),e_i )
\thicksim_{\kappaiv} \mu(\0,\mu(x_1,\ldots,x_n;\1),\mu(\0,\1,e_i))\\
&\thicksim_{\rho(L)}& \mu(\0,\mu(x_1,\ldots,x_n;\1),e_i)
\thicksim_{C_n} \mu(\mu(\0,e_i,x_1), \ldots, \mu(\0,e_i,x_n);\1 ),
\end{eqnarray*}
where the final estimate follows from Lemma \ref{coarse iterated estimate}(3). Since $x_i\in [\0,e_i]$, we have $\mu(\0,x_i,e_i)\thicksim_\kappaiv x_i$; while for $j\not=i$, we have
\[
\langle\0,e_i, x_j\rangle\thicksim_{\rho(\kappaiv)} \langle\0,e_i, \langle\0,x_j,e_j\rangle\rangle\thicksim_{\kappaiv} \langle\0,\langle e_i, \0, e_j\rangle,x_j\rangle\thicksim_{\rho(L)} \0.
\]
Hence applying Lemma \ref{coarse iterated estimate}(1)  we obtain that
\[
\mu(\mu(\0,e_i,x_1), \ldots, \mu(\0,e_i,x_n);\1 )\thicksim_{C'''} \langle\underbrace{\0, \ldots, \0}_{i-1}, x_i, \underbrace{\0,\ldots ,\0}_{n-i};\1\rangle =\langle x_i,\underbrace{\0,\ldots ,\0}_{m};\1\rangle,
\]
where if $i = 1$ then trivially $m=n-1$ and otherwise $m=n-i+1$. Here $C''':=\rho_n((n-1)(\rho(\kappaiv)+\kappaiv+\rho(L))+\kappaiv)$. Since all of these iterated medians lie in $[\0,\1]$, the cost of removing the last zero is $\kappaiv$. Hence at worst (removing $(n-2)$ zeros) we have:
\begin{eqnarray*}
\jedit{\langle x_i,\underbrace{\0,\ldots ,\0}_{n-1};\1\rangle}&\thicksim_{(n-2)\kappaiv}&\mu(x_i,\0,\1)
\thicksim_{\rho(L)}\mu(\0, \mu(\0,x_i, e_i),\1)\\
&\thicksim_\kappaiv& \mu(\0,\mu(\0,e_i,\1),x_i)\thicksim_{\rho(L)} \mu(\0, e_i, x_i)\thicksim_\kappaiv  x_i.
\end{eqnarray*}
Combining them together, we obtain that $\Phi \circ \Psi$ is $[n(3\rho(L)+(n+1)\kappaiv+C_n+C''')]$-close to the identity on $\mathcal B$.

To sum up, taking
$$C=\max\{C',C'',n(3\rho(L)+(n+1)\kappaiv+C_n+C''')\},$$
we have proved that both $\Phi$ and $\Psi$ are $\rho_+$-bornologous, $\Phi$ is a $C$-quasi-morphism and $\Phi\circ \Psi\thicksim_C id_{\mathcal B}, \Psi\circ \Phi\thicksim_C id_{\mathcal A}$. Hence by definition, $\Phi$ is a $(\rho_+,C)$-coarse median isomorphism with inverse $\Psi$.
\end{proof}

The above theorem suggests that we may regard the space $\mathcal A$ as a coarse cube (or, at least, cuboid) in our coarse median space. We now consider a natural family of subspaces, regarded as subcubes of $\mathcal A$. Given points $x_i \in [\0,e_i]$ and taking $x:=\Psi((x_1,\ldots,x_n))$ in $[\0,\1]$, we consider the following coarse median spaces:
\[
\mathcal A':=([\0,x], d, \mu_{0,x}); \quad\mathcal B':=([\0,x_1]\times \ldots \times [\0, x_n], d_{\ell^1}, \mu_{\ell^1}')
\]
where $d_{\ell^1}$ denotes the $\ell^1$-product of the induced metrics on the intervals $[\0, x_i]$ and $\mu_{\ell^1}'$ is defined by $\mu_{\ell^1}'=\mu_{\0,x_1}\times \ldots \times \mu_{\0, x_n}$. Also define maps as follows:
$$
\left.
   \begin{array}{ccccc}
      \Phi':\mathcal A'\rightarrow \mathcal B', &\quad& y & \mapsto & (\mu(\0,y,x_1), \ldots, \mu(\0,y,x_n));\\
      \Psi':\mathcal B'\rightarrow \mathcal A', &\quad&(y_1, \ldots, y_n)& \mapsto & \mu(\mu(y_1,\ldots, y_n;x),\0,x).
   \end{array}
\right.
$$

\begin{cor}\label{subcubes}
Let $(X,d,\mu)$ be a coarse median space and $f:I^n\rightarrow X$ be an $L$-coarse cube of rank $n$ in $X$. Then the map $\Phi':\mathcal A' \to\mathcal B'$ defined above provides a $(\rho'_+,C')$-coarse median isomorphism with inverse $\Psi'$, where $\rho'_+(t)=K't+H'_0$ and $K',H'_0,C'$ depend only on $n$, $L$ and parameters of $(X,d,\mu)$.
\end{cor}

\begin{proof}
It follows from the same arguments in the first part of the proof of Theorem \ref{productresult} that $\Phi',\Psi'$ are $\rho_+$-bornologous and $\Phi'$ is a $C$-coarse median morphism for the same constants $\rho_+,C$ as in Theorem \ref{productresult}. It suffices to prove that $\Psi' \circ \Phi'$ and $\Phi' \circ \Psi'$ are close to the corresponding identities.

$\bullet$~Recall that for $\Phi$ and $\Psi$, the map $\Phi \circ \Psi$ is $C$-close to the identity. So we have
$$(x_1,\ldots,x_n)\thicksim_C \Phi\circ \Psi((x_1,\ldots,x_n))=\Phi(x)=(\mu(\0,x,e_1),\ldots,\mu(\0,x,e_n)),$$
which implies that $x_i \thicksim_C \mu(\0,x,e_i)$ for each $i$. As showed in the proof of Theorem \ref{productresult}, we have $\mu(e_1, \ldots,e_n;\1)\thicksim_{H_n(L)} \1$. Combining them together with parts (1), (2) and (4) of Lemma \ref{coarse iterated estimate}, we obtain that
\begin{eqnarray*}
\mu(x_1,\ldots,x_n;x) &\thicksim_{\rho_n(nC+\kappaiv)}& \mu(\mu(\0,x,e_1),\ldots,\mu(\0,x,e_n);\mu(\0,x,\1))\\
&\thicksim_{D_n}& \mu(\0,x,\mu(e_1,\ldots,e_n;\1)) \thicksim_{\rho(H_n(L))}\mu(\0,x,\1)\thicksim_{\kappaiv}x,
\end{eqnarray*}
i.e., $\mu(x_1,\ldots,x_n;x) \thicksim_{\alpha_n(L)}x$ for $\alpha_n(L):=\rho(H_n(L))+\rho_n(nC+\kappaiv)+D_n+\kappaiv$.
Now for any $y \in [\0,x]$, we have:
\begin{eqnarray*}
  \Psi' \circ \Phi'(y) &=& \mu(\mu(\mu(\0,y,x_1), \ldots, \mu(\0,y,x_n);x),\0,x)\\
   &\thicksim_{\rho(C_n)}&\mu(\mu(\0,y,\mu(x_1,\ldots, x_n;x)),\0,x)
   \thicksim_{\rho^2(\alpha_n(L))}\mu(\mu(\0,y,x),\0,x) \thicksim_{2\kappaiv} y.
\end{eqnarray*}
Hence $ \Psi' \circ \Phi'$ is $C''$-close to $\mathrm{Id}_{\mathcal A'}$ for $C'':=\rho^2(\alpha_n(L))+\rho(C_n)+2\kappaiv$.

$\bullet$~For the other direction, $x_i \thicksim_C \mu(\0,x,e_i)$ implies:
$$\mu(\0,x_i,x)\thicksim_{\rho(C)}\mu(\0,\mu(\0,x,e_i),x)\thicksim_\kappaiv \mu(\0,x,e_i)\thicksim_C x_i.$$
Hence for any $\vec{y}=(y_1, \ldots, y_n) \in [\0,x_1]\times \ldots \times [\0, x_n]$, we have
\begin{eqnarray*}
pr_i \circ \Phi' \circ \Psi' (\vec{y}) &=& \mu(\0,\mu(\mu(y_1,\ldots,y_n;x),\0,x),x_i )\\
&\thicksim_{\kappaiv}& \mu(\0,\mu(y_1,\ldots,y_n;x),\mu(\0,x,x_i))\thicksim_{\rho(\rho(C)+C+\kappaiv)} \mu(\0,\mu(y_1,\ldots,y_n;x),x_i)\\
&\thicksim_{C_n}& \mu(\mu(\0,x_i,y_1), \ldots, \mu(\0,x_i,y_n);x ),
\end{eqnarray*}
where the final estimate follows from Lemma \ref{coarse iterated estimate}(3).

On the other hand, since $\langle e_i,\0,e_j\rangle \thicksim_L \0$ for $i\neq j$, we have
$$\langle x_i,\0,e_j\rangle\thicksim_{\rho(\kappaiv)}\langle\langle\0,x_i,e_i\rangle,\0,e_j\rangle\thicksim_{\kappaiv}\langle\0,x_i,\langle e_i,\0,e_j\rangle\rangle\thicksim_{\rho(L)}\langle\0,x_i,\0\rangle=\0.$$
This implies that
\begin{eqnarray*}
\langle x_i,\0,x_j\rangle\thicksim_{\rho(\kappaiv)}\langle x_i,\0,\langle\0,x_j,e_j\rangle\rangle
\thicksim_\kappaiv \langle\0,x_j,\langle x_i,\0,e_j\rangle\rangle\thicksim_{\rho(\rho(L)+\rho(\kappaiv)+\kappaiv)}\langle\0,x_j,\0\rangle=\0.
\end{eqnarray*}
In other words, $\langle x_i,\0,x_j\rangle \thicksim_{\beta_n(L)}\0$ for $\beta_n(L):=\rho(\rho(L)+\rho(\kappaiv)+\kappaiv)+\rho(\kappaiv)+\kappaiv$. Notice that $\mu(\0,y_i,x_i)\thicksim_\kappaiv y_i$, so for $j\not= i$ we have
$$\langle\0,x_i, y_j\rangle\thicksim_{\rho(\kappaiv)} \langle\0,x_i, \langle\0,y_j,x_j\rangle\rangle\thicksim_{\kappaiv} \langle\0,\langle x_i, \0, x_j\rangle,y_j\rangle\thicksim_{\rho(\beta_n(L))} \0.$$
Now using the same arguments as in the proof of Theorem \ref{productresult}, we obtain that for the constant
$$C''':=\rho_n((n-1)\rho(\beta_n(L))+(n-1)\rho(\kappaiv)+n\kappaiv),$$
we have
\begin{eqnarray*}
\mu(\mu(\0,x_i,y_1), \ldots, \mu(\0,x_i,y_n);x )\thicksim_{C'''} \mu(\0, \ldots, \0, y_i, \0,\ldots ,\0;x)\thicksim_{(n-2)\kappaiv}\mu(\0,y_i,x)\\
\thicksim_{\rho(\kappaiv)}\mu(\0, \mu(\0,y_i, x_i),x)\thicksim_\kappaiv \mu(\0,\mu(\0,x_i,x),y_i)\thicksim_{\rho(\rho(C)+C+\kappaiv)} \mu(\0, x_i, y_i)\thicksim_\kappaiv  y_i.
\end{eqnarray*}
Therefore $\Phi' \circ \Psi'$ is $D'$-close to $\mathrm{Id}_{\mathcal B'}$ for
$$D':=n[C'''+2\rho(\rho(C)+C+\kappaiv)+\rho(\kappaiv)+(n+1)\kappaiv+C_n].$$
Finally setting $\rho_+'=\rho_+$ and $C'=\max\{C,C'',nD'\}$, we finish the proof.
\end{proof}

\subsection{Rank and coarse interval growth}\label{growth}
In this subsection, we will give a characterisation of rank in terms of interval growth as a converse to a result of Bowditch from \cite{bowditch2014embedding}.

First we notice that the cardinality of intervals can always be bounded in terms of the distance between its endpoints in the context of bounded geometry coarse median spaces:

\begin{lem}\label{finiteness}
Let $(X,d,\mu)$ be a coarse median space with parameters $(\rho, \kappaiv, \kappav)$. If $a,b\in X$ with $d(a,b) \leqslant r$, then $[a,b] \subseteq B(a,\rho(r))$. If in addition $(X,d)$ has bounded geometry, then there exists a constant $C(r)$ such that $\card [a,b] \leqslant C(r)$.
\end{lem}

\begin{proof}
For any $c\in [a,b]$, there exists some $x\in X$ such that $c=\mu(a,b,x)$. Now by axiom (C1), we have
$$c=\mu(a,b,x)\thicksim_{\rho(r)}\mu(a,a,x)=a,$$
which implies $c\in B(a,\rho(r))$. The second statement follows directly by the definition of bounded geometry.
\end{proof}

For the remainder of this section, we will specialise to the context of \emph{uniformly discrete quasi-geodesic} coarse median spaces with bounded geometry. 
Recall that for a metric space $(X,d)$ and $C>0$, the \emph{Rips complex} is the simplicial complex, in which $\sigma=[x_0,x_1,\ldots,x_n]$ is an $n$-simplex for $x_0,x_1,\ldots,x_n \in X$ if and only if $d(x_i,x_j) \leqslant C$. 

Bowditch proved in \cite{bowditch2014embedding} that for a uniformly discrete coarse median space of bounded geometry and finite rank, there is a polynomial bound on growth within intervals. Now given an interval $[a,b]$ in such a space $X$ with parameters $(K,H_0,\kappaiv,\kappav)$, any point $x\in [a,b]$ can be written in the form  $x=\mu(a,y,b)$. Hence
$$x=\mu(a,y,b)\thicksim_{Kd(a,b)+H_0} \mu(a,y,a) =a,$$
which implies that $\diam([a,b]) \leqslant 2Kd(a,b)+2H_0$.
Taking the subset $Q=[a,b]\subseteq [a,b]_\kappaiv$ (where $[a,b]_\kappaiv$ is Bowditch's definition of coarse interval), we obtain the following as a corollary to Bowditch's result \cite[Proposition 9.8]{bowditch2014embedding}.

\begin{prop}\label{bounding intervals in terms of diam}
Let $(X,d,\mu)$ be a uniformly discrete quasi-geodesic coarse median space with bounded geometry and rank at most $n$. Then there is a function $p: \mathbb N \to \mathbb N$ with $p(r)=o(r^{n+\epsilon})$ for all $\epsilon>0$, such that $\card[a,b] \leqslant p(d(a,b))$ for any $a,b\in X$.
\end{prop}

\begin{proof}
Bowditch proved this result in the context that $X$ is a connected bounded valency graph with edge-path metric. We replace our metric $d$ with the edge-path metric provided by the Rips complex. Since the metric space $(X,d)$ is quasi-geodesic and has bounded geometry, taking the Rips parameter sufficiently large ensures that this metric provides $X$ with the structure of a connected bounded valency graph as required. Furthermore this new metric $d'$ is quasi-isometric to $d$, again using the fact that the space is quasi-geodesic. Applying Bowditch's result, $\card[a,b]$ is $o(r^{n+\epsilon})$ where $r=d'(a,b)$. Since $d$ is $O(d')$ the result follows.
\end{proof}

We now provide a converse to Bowditch's theorem, showing that this growth condition indeed characterises the rank.

\begin{thm}\label{coarseintervalrank}\label{growth rank}
\growthrankthm
\end{thm}

\begin{proof}[Proof of Theorem \ref{coarseintervalrank}]
(1)$\Rightarrow$(2) is given by Proposition \ref{bounding intervals in terms of diam}, while (2)$\Rightarrow$(3) \emph{a fortiori}.

(3)$\Rightarrow$(1): Suppose $X$ is $(\alpha,\beta)$-quasi-geodesic, $(K,H_0, \kappaiv, \kappav)$ are parameters of $X$ and $\rank X>n$ (note that we do not assume $X$ has finite rank). By Theorem \ref{char for high rank-final}, there exists a constant $L_0>0$, such that for any $C>0$, there exists an $L_0$-coarse cube $\sigma: I^{n+1} \rightarrow X$ with $d(\sigma(\bar{e}_i),\sigma(\bar{\0}))>C$ for all $i$. After setting $\0:=\sigma(\bar{\0}), \1:=\sigma(\bar{\1})$ and $e_i:=\sigma(\bar{e}_i)$ for each $i$, we have $d(e_i,\0)>C$.

Now choose a discrete $(\alpha, \beta)$-quasi-geodesic $\0=p_0, \ldots p_k=e_i$ and construct $q_j=\langle\0,p_j,e_i\rangle$ to get a sequence of points in $[\0, e_i]$ with $d(q_j, q_{j-1})\leq G$ where $G=K(\alpha+\beta)+H_0$ is independent of $C$. Since $d(\0, q_0)=0$ and $d(\0, q_k)>C$,  we may choose the first $j$ such that $d(\0, q_j)\geq C$ and for this $j$ we also have $d(\0, q_j)<C+G$. Setting $x_i:=q_j\in [\0, e_i]$, we have $C\leqslant d(\0, x_i)<C+G$.

Choose a discrete $(\alpha,\beta)$-quasi-geodesic $z_0,z_1,\ldots,z_k\in X$ connecting $\0$ and $x_1$. Projecting $z_i$ into $[\0,x_1]$, we obtain a sequence $\0=y_0,y_1,\ldots,y_k=x_1$ with $d(y_i,y_{i-1}) \leqslant \jedit{G}$, where $y_i=\mu(\0,z_i,x_1)$. We will inductively ``de-loop'' this sequence to define a subsequence $y_{j_0},\ldots,y_{j_l}$ such that the points in it are distinct, but still satisfy $d(y_{j_p},y_{j_{p-1}}) \leqslant \jedit{G}$. Let $j_0$ be the maximal index such that $y_{j_0}=y_0$. Then for $l>0$, set $j_p$ to be the maximal index such that $y_{j_p}=y_{j_{p-1}+1}$ to obtain the required sequence. This process allows us to assume that we have picked the sequence $\0=y_0,y_1,\ldots,y_l=x_1$ to be distinct while ensuring that $d(y_i,y_{i-1}) \leqslant \jedit{G}$ for each $i$. Now we have:
$$C\leqslant d(\0,x_1) \leqslant \sum_{i=1}^l d(y_i,y_{i-1}) \leqslant l\cdot \jedit{G}, $$
which implies $\card[\0,x_1] \geqslant l\geqslant \jedit{CG^{-1}}$. Similar estimate holds for each $[\0,x_i]$. Hence we obtain that for the constant $\gamma:=\jedit{G}^{-(n+1)}$,
$$\card([\0,x_1]\times \ldots \times [\0,x_{n+1}]) \geqslant \gamma C^{n+1}.$$

Now set $x:= \mu(\0, \mu(x_1, \ldots, x_{n+1}; \1),\1)$. By Corollary \ref{subcubes}, there exists:
\begin{itemize}
\item a constant $\lambda_0:=\max\{K', H'_0, C'\}$,  depending only on $n,L_0$ and parameters of the space;
\item a $(\lambda_0t+\lambda_0,\lambda_0)$-coarse median isomorphism
\[
\Psi':[\0,x_1]\times \ldots \times [\0,x_{n+1}] \to [\0,x].
\]
\end{itemize}
 In particular for any $\vec{z},\vec{y} \in [\0,x_1]\times \ldots \times [\0,x_{n+1}]$ we have:
\begin{equation}\label{distance bounds original}
 \lambda_0^{-1}d_{\ell^1}(\vec{z},\vec{y})-\lambda_0 \leqslant d(\Psi'(\vec{z}),\Psi'(\vec{y})) \leqslant \lambda_0d_{\ell^1}(\vec{z},\vec{y})+\lambda_0.
\end{equation}
Since $X$ has bounded geometry, there exists a constant $N$ depending only on $\lambda_0$ such that $\card\Psi'^{-1}(\{y\}) \leqslant N$ for any $y\in [\0,x]$. In other words, $\Psi'$ may collapse at most $N$ points to a single point. Hence $\card\Psi'(A) \geqslant \frac{1}{N}\card A$ for any $A\subseteq [\0,x_1]\times \ldots \times [\0,x_{n+1}]$. In particular, we have
\begin{equation}\label{intervalcardinalityestimate}
\card[\0,x] \geqslant \card\Psi'([\0,x_1]\times \ldots \times [\0,x_{n+1}]) \geqslant \frac{1}{N} \card([\0,x_1]\times \ldots \times [\0,x_{n+1}]) \geqslant \frac{\gamma}{N} C^{n+1}.
\end{equation}

Now we would like to estimate the distance $d(\0,x)$ and show that it is approximately linear in $C$. First notice that $\Psi'(\vec{0})=\0$ and by definition we have
\begin{eqnarray*}
\Psi'(\vec{x})&=&\mu(\mu(x_1,\ldots, x_{n+1};x),\0,x)=\mu(x_1,\ldots, x_{n+1},\0;x)\\
&=& \mu(x_1,\ldots, x_{n+1},\0;\mu(x_1,\ldots, x_{n+1},\0;\1))\\
&\thicksim_{E_n}& \mu(x_1,\ldots, x_{n+1},\0;\1)=x
\end{eqnarray*}
where the estimate in the third line follows from Lemma \ref{coarse iterated estimate new2} and the constant $E_n$ depends only on $n,\lambda_0$ and $\kappaiv$. Combining with (\ref{distance bounds original}), we have:
\begin{eqnarray*}
d(\0,x) &\leqslant & d(\Psi'(\vec{0}),\Psi'(\vec{x}))+E_n \leqslant \lambda_0 d_{\ell^1}(\vec{0},\vec{x})+\lambda_0+E_n=\lambda_0 \sum\limits_{i=1}^{n+1}d(\0, x_i)+\lambda_0 + E_n  \\
&\leqslant & \lambda_0 (n+1)(C+G)+\lambda_0 + E_n.
\end{eqnarray*}
After rearranging, we get
\[
C\geqslant\frac{d(\0,x)-\lambda_0(nG+G+1)-E_n}{\lambda_0 (n+1)}.
\]
Combining with (\ref{intervalcardinalityestimate}), we obtain:
$$\card[\0,x] \geqslant \frac{\gamma}{N} \Big(\frac{d(\0,x)-\lambda_0(nG+G+1)-E_n}{\lambda_0 (n+1)}\Big)^{n+1}.$$
On the other hand, (\ref{distance bounds original}) implies that
\begin{eqnarray*}
d(\0,x)&\geqslant & d(\Psi'(\vec{0}),\Psi'(\vec{x}))-E_n \geqslant \lambda_0^{-1}d_{\ell^1}(\vec{0},\vec{x})-\lambda_0-E_n \\
&\geqslant &\lambda_0^{-1}(n+1)C-\lambda_0-E_n.
\end{eqnarray*}
So $d(\0,x) \to \infty$ as $C \to \infty$.

Therefore for any $C>0$ we have constructed an interval $[\0,x]$ such that the distance $d(\0,x)$ goes to infinity as $C \to \infty$, and the cardinality $\sharp[\0,x]$ is bounded below by a polynomial of degree $n+1$ in $d(\0,x)$ with positive leading coefficient $\frac{\gamma}{N(\lambda_0(n+1))^{n+1}}$. This contradicts the existence of the function $p$.
\end{proof}

Theorem \ref{coarseintervalrank} allows us to characterise the rank of a coarse interval space purely in terms of the growth of intervals:

\begin{cor}
A uniformly discrete, bounded geometry, quasi-geodesic coarse interval space $(X, d,[\cdot,\cdot])$ has rank at most $n$ \emph{if and only if} there is a function $p: \Rp \to \Rp$ with $\lim\limits_{r\rightarrow\infty}p(r)/r^{n+1}=0$, such that $\card[a,b] \leqslant p(d(a,b))$ for any $a,b\in X$.
\end{cor}

\section{Intervals and metrics for ternary algebras}\label{coarsemedianalgebras}

Bowditch observed that perturbing the metric for a coarse median space up to quasi-isometry respects the coarse median axioms. It is not, however, \emph{a priori} obvious the extent to which the metric is determined by the coarse median operator.  We will now show that for a quasi-geodesic coarse median space $(X,d,\mu)$ of bounded geometry the metric is  determined \emph{uniquely} up to quasi-isometry by $\mu$. This motivates our definition of coarse median algebra, as given in the introduction.

To establish the uniqueness of the metric, we will construct a canonical metric defined purely in terms of the intervals associated to the coarse median operator. The construction may be of independent interest since it can be defined for any ternary operator satisfying some weakening of axioms (M1) and (M2), and therefore in the context of a more general notion of interval structure. (The following reversal axiom can in fact be weakened to the existence of bijections between the corresponding intervals $[a,b]$ and $[b,a]$.)

\subsection{Abstract ternary algebras and induced metrics}\label{ternary}
\let\nu=\mu
Consider a ternary algebra $(X,\nu)$ satisfying the following axioms:

\begin{itemize}
  \item[(T1) \emph{Majority vote}:] $\nu(a,a,x)=\nu(a,x,a)=a$ for all $a,x\in X$;
  \item[(T2) \emph{Reversal}:] $\nu(a,x,b)=\nu(b,x,a)$ for all $a,x,b\in X$.
\end{itemize}

Classically it is natural to think of the ternary operator $\nu$ as furnishing a notion of betweenness, whereby $c$ lies between $a, b$ if and only if $\nu(a,c,b)=c$. This definition is not well adapted to the coarse world, where statements are typically true up to controlled distortion. Regarding the operation $x\mapsto \nu(a,x,b)$ instead as providing a projection onto the interval $[a,b]=\{\mu(a,x,b)\mid x\in X\}$ is better suited to this environment.

Axiom (T1) ensures that the interval $[a,a]$ is the singleton $\{a\}$ while axiom (T2) ensures that $[a,b]=[b,a]$. These axioms together are a slight weakening of axioms (M1) and (M2) for a (coarse) median algebra.

\begin{ex}\label{graph}
Let $\Gamma$ be a connected graph and for any $a,b,x\in V(\Gamma)$ choose a vertex, denoted $\nu(a,x,b)$, which lies on an edge geodesic from $a$ to $b$ and minimises distance to $x$ among all such choices. Clearly we can do so to satisfy axiom (T2), while axiom (T1) is immediate. With this definition of the ternary operator, the interval $[a,b]$ is exactly the set of vertices on edge geodesics from $a$ to $b$.
\end{ex}

We will use cardinalities of intervals to measure distances.  In order to ensure that these distances are finite, we need to impose a condition that points can be joined by chains of finite intervals:

\begin{defn}
A ternary algebra $(X,\nu)$ is said to satisfy the \emph{finite interval chain condition}, if for any $a,b\in X$ there exists a sequence $a=x_0, x_1, \ldots, x_n:=b$  in $X$ such that the cardinality of each interval $[x_i, x_{i+1}]$ is finite for $i=0,1,\ldots, n-1$.
\end{defn}

\begin{defn}
Given a ternary algebra $(X,\nu)$ satisfying the finite interval chain condition, we define the \emph{induced function} $d_\nu$ on $X\times X$ as follows: for any $a,b \in X$,
\[
d_\nu\relax (a,b)=\min \Big\{ \sum_{i=1}^n (\card [x_{i-1},x_i]-1): a=x_0,\ldots,x_n=b, x_i\in X, n\in \mathbb{N} \Big\}.
\]
\end{defn}

It is routine  to check that $d_\nu$  satisfies the triangle inequality. The imposition of axioms (T1) and (T2) ensure that the function $d_\nu$ also satisfies the obvious symmetry, reflexivity and positivity conditions so that $d_\nu$ is  a metric in this case. When (T1) and (T2) are satisfied we will refer to $d_\nu$ as the \emph{induced metric}.

\begin{ex}
Let $(X,\mu)$ be a discrete median algebra and let $Z$ be its geometric realisation as a CAT(0) cube complex. Then the induced metric $d_\mu$ is the edge-path metric on the vertices of $Z$.
\end{ex}

\begin{ex}
Let $\Gamma$ be a connected graph and $\nu$ the projection operator defined in Example \ref{graph}. Then the induced metric $d_\nu$ is the edge-path metric on the vertices of $\Gamma$.
\end{ex}

\subsection{Uniqueness of coarse median metrics}\label{uniquemetricsection}

It is easy to show that one can change the metric of a coarse median space arbitrarily within its quasi-isometry class. It is a remarkable fact, as we will now show, that the quasi-isometry class of the metric is determined uniquely by the coarse median operator. Indeed the induced metric is the unique coarse median metric up to quasi-isometry:

\begin{thm}\label{d qi to d_mu}
For a bounded geometry quasi-geodesic coarse median space $(X,d,\mu)$, the metric $d$ is quasi-isometric to the induced metric $d_\mu$.
\end{thm}

As an immediate corollary we have the following:

\begin{thmuniquemetric}
For a bounded geometry quasi-geodesic coarse median space $(X,d,\mu)$, the metric $d$ is unique up to quasi-isometry.
\end{thmuniquemetric}

\begin{proof}[Proof of Theorem \ref{d qi to d_mu}]
Let $(X,d,\mu)$ be an $(L,C)$-quasi-geodesic coarse median space with bounded geometry and parameters $(K, H_0, \jedit{\sout{\kappao,}}\kappaiv,\kappav)$.

{First} we will show that $d$ can be controlled by $d_\mu$. Given $a,b\in X$,  let $a=a_0,\ldots,a_n=b$ be a sequence of points such that
$$d_{\mu\relax}(a,b)=\sum_{i=1}^n (\card [a_{i-1},a_i]-1).$$
Fix $i$ and choose an $(L,C)$-quasi-geodesic $\gamma_i$ with respect to the metric $d$ connecting $a_{i-1}$ and $a_i$. If $n_i=\lfloor d(a_{i-1},a_i) \rfloor$, the integer part of $d(a_{i-1},a_i)$, and
$$x_0=\gamma_i(0)=a_{i-1},x_1=\gamma_i(1),\ldots,x_{n_i}=\gamma_i(n_i), x_{n_i+1}=\gamma_i(d(a_{i-1},a_i))=a_i,$$
then $d(x_{i-1},x_i) \leqslant L+C$. Letting $y_j=\langle a_{i-1},a_i,x_j\rangle\in [a_{i-1},a_i]$, axiom (C1) ensures that  $d(y_{j-1},y_j) \leqslant K(L+C)+H_0$. We set $C':=K(L+C)+H_0$.

As in the proof of Theorem \ref{coarseintervalrank} we can ``de-loop'' the sequence $y_0,y_1,\ldots,y_{n_i+1}$ to a subsequence $y_{j_0},\ldots,y_{j_l}$ in $[a_{i-1},a_i]$ with the property that the points in it are distinct, but still satisfy $d(y_{j_k},y_{j_{k-1}}) \leqslant C'$.
Hence we have
$$d(a_{i-1},a_i) \leqslant \sum_{k=1}^l d(y_{j_{k-1}},y_{j_k}) \leqslant l \cdot C' \leqslant (\card [a_{i-1},a_i]-1) \cdot C'.$$
The same estimate holds for other $i$ as well. Therefore we obtain that
$$d(a,b) \leqslant \sum_{i=1}^n d(a_{i-1},a_i) \leqslant C' \cdot \sum_{i=1}^n (\card [a_{i-1},a_i]-1) = C' \cdot d_\mu\relax(a,b).$$

{Second} we will show that $d_\mu\relax$ can be controlled by $d$. For any $a,b\in X$ choose an $(L,C)$-quasi-geodesic $\gamma$ with respect to the metric $d$ connecting them, and take $a_i=\gamma(i)$ for $i=0,1,\ldots,n-1=\lfloor d(a,b) \rfloor$ and $a_n=\gamma(d(a,b))$, which implies $d(a_{i-1},a_i) \leqslant L+C$. By Lemma \ref{finiteness} there exists a constant $C''$ (depending on $L+C$) such that the intervals $[a_{i-1},a_i]$ all have cardinality at most $C''$. Hence we have
$$d_\mu\relax(a,b) \leqslant \sum_{i=1}^n (\card [a_{i-1},a_i]-1) < \sum_{i=1}^n C" \leqslant C" \cdot (d(a,b) + 1).$$
In conclusion we have shown that for any $a,b\in X$,
\[
\frac{1}{C'}\cdot d(a,b) \leqslant d_\mu\relax(a,b) < C" \cdot d(a,b) + C".
\]
This completes the proof.
\end{proof}

Without the assumption that $(X,d)$ is quasi-geodesic,  Theorem \ref{bi-lip equi} fails. Indeed $(X,d)$ can have bounded geometry and $(X,d_\mu)$ have balls of infinite cardinality as the following example shows:

\begin{ex}\label{F infty}
Let $F_\infty$ be the free group on countably many generators $\{g_i\}$. The Cayley graph of $F_\infty$ is a tree and therefore the group admits a median $\mu$. Note that the induced metric $d_\mu$  is the edge-path metric on the Cayley graph. With this metric $F_\infty$ is a coarse median space which does not have bounded geometry since each of the intervals $[e, g_i]$ has cardinality $2$. However, for $d$ a proper left invariant metric on $F_\infty$ (e.g., setting $d(g_i,e)=i$), the space $(F_\infty,d,\mu)$ is again a coarse median space. With this metric the space has bounded geometry. Hence $\mu$ admits two coarse median metrics which are not quasi-isometric.
\end{ex}

\begin{rem}
If we just focus on uniformly discrete metrics, then it is clear that ``quasi-isometry" can be replaced by ``bi-Lipschitz equivalence" in Theorem \ref{bi-lip equi}.
\end{rem}

\section{Coarse median algebras}\label{coarsemedianalg}

We have seen that intervals play a key role in determining the structure and geometry of a coarse median space. In particular, as shown in Theorem \ref{uniquemetricthm}, for a quasi-geodesic coarse median space of bounded geometry the metric is  determined by the interval structure and is therefore redundant in the description. This leads us to the following purely algebraic notion of coarse median algebra.

\begin{defcma}
A  \emph{coarse median algebra} is a ternary algebra $(X,\mu)$ with finite intervals such that:
\begin{itemize}
\item[(M1)\phantom{'}] For all $a,b\in X$,  $\mu(a,a,b)=a$;
\item[(M2)\phantom{'}] For all $a,b,c\in X$, $\mu(a,b,c)=\mu(a,c,b)=\mu(b,a,c)$;
\item[(M3)'] There exists a constant $K\geq 0$ such that for all $a,b,c,d,e\in X$ the cardinality of the interval
\;$\big[\mu(a,b,\mu(c,d,e)),\, \mu(\mu(a,b,c),\mu(a,b,d),e)\big]$\; is at most $K$.
\end{itemize}
\end{defcma}

As remarked in the introduction if we take $K=1$ then this reduces to the classical definition of a discrete median algebra.

\subsection{Bounded geometry for a ternary algebra}

\begin{defn}\label{bounded valency def}
A ternary algebra  $(X,\nu)$  is said to have \emph{bounded valency} if there is a function $\phi:\mathbb R^+\rightarrow \mathbb R^+$ such that for all $x\in X$, we have
\[\sharp \{y\in X\mid \sharp [x,y]\leq R\}\leq \phi(R).\]
\end{defn}

The terminology is motivated by the example of a median graph, where bounded valency in our sense agrees with its classical meaning.

\begin{lem}\label{bdd geo}
Let $(X,\nu)$ be a ternary algebra satisfying (T1) and (T2) together with the finite interval chain condition. Then it has bounded valency \emph{if and only if} the induced metric $d_\nu$ has bounded geometry.
\end{lem}

\begin{proof}
Fix $x\in X$ and $R>1$. Since $d_\nu\relax(x,y) \leq \sharp[x,y]-1$, we have
$$\{y\in X\mid \sharp [x,y]\leq R\} \subseteq B_{R-1}(x).$$
Hence bounded geometry of $d_\nu$ implies bounded valency. On the other hand, suppose $X$ has bounded valency with parameter $\phi$. For any $y\in B_R(x)$ there is an interval chain $x=x_0, \ldots , x_n=y$ with $n\leq R$ and such that each interval $[x_i, x_{i+1}]$ has at most $R+1$ points. It follows that given $x_i$ the number of possible choices for $x_{i+1}$ is at most $\phi(R+1)$, so $B_R(x)$ has cardinality at most $\phi(R+1)^R$.
\end{proof}

\begin{thmcmacmsa}
 \cmacmsthma
\end{thmcmacmsa}

\begin{proof}
$(1)\Rightarrow (2)$: Suppose $(X,\mu)$ is a bounded valency coarse median algebra. We impose the induced metric $d_\mu$, which has bounded geometry by Lemma \ref{bdd geo}. Axiom (M3)' gives us an upper bound on the distance between the two iterated medians, $\mu(a,b,\mu(x,y,z))$ and $\mu(\mu(a,b,x),\mu(a,b,y),z)$, which specialises to the 4-point axiom (C2) by setting $y=b$. It only remains to establish axiom (C1).

To do so, we choose a finite interval chain  $a=x_0, \ldots , x_n=a'$ which realises the distance $d_\mu{}(a, a')$. For each $i$, let $y_i=\mu(x_i,b,c)$ and consider the interval chain $y_0=\mu(a,b,c), \ldots, y_n=\mu(a',b,c)$ which gives an upper bound for $d_\mu{}(\mu(a,b,c),  \mu(a',b,c))$. For each point
\[
\mu(z, y_i, y_{i+1}) = \mu(z,\mu(x_i, b,c),\mu(x_{i+1},b,c))
\]
in the interval $[y_i, y_{i+1}]$, the interval from $\langle z, y_i, y_{i+1}\rangle$ to $\mu(\mu(z,x_i, x_{i+1}),b,c)$ has cardinality at most $K$ by axiom (M3)'. Clearly the set $\{\mu(\mu(z,x_i, x_{i+1}),b,c)\mid z\in X\}$ has cardinality bounded by the cardinality of $[x_i, x_{i+1}]$. So by bounded valency, the interval $[y_i, y_{i+1}]$ has cardinality bounded by $\phi(K) \cdot \card[x_i,x_{i+1}]$. It follows that
\begin{align*}
d_\mu{}(\mu(a,b,c), \mu(a',b,c)) &\jedit{\leqslant \sum_{i=0}^{n-1}(\card[y_i,y_{i+1}]-1)} \leqslant \phi(K)\sum_{i=0}^{n-1}\card[x_i,x_{i+1}]\\
&\leqslant \jedit{2\phi(K)\sum_{i=0}^{n-1}(\card[x_i,x_{i+1}]-1)} = 2\phi(K)d_\mu{}(a,a').
\end{align*}
Therefore $(X,d_\mu{},\mu)$ is a coarse median space.

$(2)\Rightarrow (3)$:  This is trivial.

$(3)\Rightarrow (1)$: Suppose there exists a bounded geometry metric $d$ on $X$ such that $(X,d,\mu)$ is a coarse median space. As remarked after Definition \ref{def:coarse median space}, the five point condition in Equation (\ref{five point estimate}) holds up to some constant $\kappav'$. Hence Lemma \ref{finiteness} implies that (M3)' holds for the constant $\kappav=C(\kappav')$ where $C$ is the function provided therein. Therefore $(X,\mu)$ is a coarse median algebra.
\end{proof}

\begin{rem}
While it is tempting to conflate the ideas of bounded geometry and bounded valency in this context, some care should be taken in the general world of coarse median spaces.  In this context the metric is only loosely associated with the median structure as illustrated by Example \ref{F infty}: the free group $F_\infty$, equipped with a proper left invariant metric and its natural median, is a coarse median space which has bounded geometry but not bounded valency. Of course this example is not quasi-geodesic and as we saw in Theorem \ref{uniquemetricthm} we have much better control in the quasi-geodesic world.
\end{rem}

\subsection{Quasi-geodesic ternary algebras}

\begin{defn}\label{qgdef}
A ternary algebra  $(X,\nu)$ satisfying (T1) and (T2)  is said to be \emph{quasi-geodesic} if there exist constants $L,C>0$ such that for any $a,b\in X$, there exist $a=y_0,\ldots,y_n=b$ with $\sharp [y_j,y_{j+1}] \leqslant C+1$ and $n \leqslant L\sharp [a,b]$.
\end{defn}

Note that the finite interval chain condition is subsumed in this definition so does not need to be imposed separately.

This definition has a natural interpretation in the terms of the following analogue of the classical Rips Complex.

\begin{defn}
For $(X,\nu)$ a ternary algebra, let $P_C(X, \nu)$ denote the simplicial complex in which $\sigma=[x_0,x_1,\ldots,x_n]$ is an $n$-simplex for $x_0,x_1,\ldots,x_n \in X$ if and only if $\sharp[x_i,x_j] \leqslant C+1$.
\end{defn}

Recall for comparison that if $(X,d)$ is a metric space then for $C>0$ the \emph{Rips complex} is the simplicial complex, in which $\sigma=[x_0,x_1,\ldots,x_n]$ is an $n$-simplex for $x_0,x_1,\ldots,x_n \in X$ if and only if $d(x_i,x_j) \leqslant C$.

When the complex $P_C(X, \nu)$ is connected, its vertex set $X$ inherits the edge-path metric which is of course a geodesic metric, denoted  $d_{P_C}$.

\begin{prop}\label{quasi-geodesic}
Let $(X,\nu)$ be a ternary algebra satisfying conditions (T1) and (T2) together with the finite interval chain condition. Let $d_\nu$ denote the induced metric . Then the following are equivalent:
\begin{enumerate}[(1)]
  \item The metric $d_\nu$ is quasi-geodesic;
  \item The ternary algebra $(X,\mu)$ is quasi-geodesic;
  \item There exists $C>0$ such that the complex $P_C(X, \nu)$ is connected and $d_\nu$ is bi-Lipschitz to the edge-path metric $d_{P_C}$ on the complex.
\end{enumerate}
\end{prop}

\begin{proof}
\emph{(1) $\Rightarrow$ (2)}:
Assume $d_\nu$ is $(L',C')$-quasi-geodesic and $a\neq b\in X$. Let $\gamma: [0,m] \rightarrow X$ be an $(L',C')$-quasi-isometric embedding with $\gamma(0)=a$ and $\gamma(m)=b$. Without loss of generality we may take $m$ to be an integer. Let $x_i=\gamma(i)$ and note that $d_\nu\relax(x_i,x_{i+1})\leq C:=L'+C'$. On the other hand $\frac 1{L'} m -C'\leq d_\nu\relax(a,b)$ so $m\leq L'd_\nu\relax(a,b)+L'C\leq L''d_\nu\relax(a,b)$, where $L''=L'+L'C'$.

Now fix $i$ and take a chain $y_i^0,\dots y_i^{n_i}$ realising the distance from $x_i$ to $x_{i+1}$, i.e.,
$$d_\nu\relax(x_i,x_{i+1})=\sum_{j=0}^{n_i-1}(\sharp [y_i^j,y_i^{j+1}]-1).$$
Since $d_\nu\relax(x_i,x_{i+1})\leq C$ it follows that each set $[y_i^j,y_i^{j+1}]$ has cardinality at most $C+1$.

Furthermore, without loss of generality, we may assume that $y_i^j\neq y_i^{j+1}$ for each $j$, which implies $n_i \leq d_\nu\relax(x_i,x_{i+1})\leq C$. Concatenating these chains gives the required chain from $a$ to $b$. Putting $L=CL''$, the number of terms is:

$$\sum_{i=0}^{m-1} n_i\leq Cm\leq CL''d_\nu\relax(a,b) < L\sharp [a,b].$$

\emph{(2) $\Rightarrow$ (3)}:
Assuming $(X,\mu)$ is $(L,C)$-quasi-geodesic, the Rips complex $P_C(X, \nu)$ is connected. If $d_{P_C}(a,b)=n$ then there exist $x_0=a,x_1,\dots,x_n=b$ with each interval $[x_{i-1},x_i]$ having cardinality at most $C+1$, and hence
$$d_\nu{}(a,b)\leq nC=Cd_{P_C}(a,b).$$

Now we fix $a,b\in X$ and choose mutually different points $a=z_0, z_1, \ldots, z_{k-1}, z_k=b$ in $X$ such that
$$d_\nu\relax(a,b)= \sum_{i=0}^{k-1}(\sharp [z_i,z_{i+1}]-1).$$
For each $i=0,1,\ldots,k-1$, applying condition (2) to $z_i, z_{i+1}$ produces a number $k_i \in \mathbb N$ and points $z_i=w_i^0,w_i^1,\ldots,w_i^{k_i-1}, w_i^{k_i}=z_{i+1}$ in $X$ with $\sharp [w_i^j,w_i^{j+1}] \leqslant C+1$ and $k_i \leqslant L\sharp [z_i, z_{i+1}]$. Since $\sharp [z_i, z_{i+1}] \geq 2$, we have $\sharp [z_i, z_{i+1}] \leq 2(\sharp [z_i, z_{i+1}]-1)$. Hence
$$p:=\sum_{i=0}^{k-1}k_i \leq L\sum_{i=0}^{k-1}\sharp [z_i, z_{i+1}] \leq 2L\sum_{i=0}^{k-1}(\sharp [z_i, z_{i+1}]-1) = 2L d_\nu\relax(a,b).$$
Concatenating these chains provides a chain $a=w_0,w_1,\ldots,w_p=b$ with $\sharp [w_i,w_{i+1}] \leq C+1$ and $p \leq 2L d_\nu\relax(a,b)$, which gives an upper bound
$$d_{P_C}(a,b)\leq p\leq 2L d_\nu\relax(a,b).$$

\emph{(3) $\Rightarrow$ (1)}: As $d_{P_C}$ is geodesic it follows that $d_\nu$ is quasi-geodesic.
\end{proof}

Combining Theorem \ref{unique metric prop} with Proposition \ref{quasi-geodesic} and Theorem \ref{bi-lip equi}, we obtain:

\begin{thm}\label{unique metric prop 2}
A ternary algebra is a bounded valency quasi-geodesic coarse median algebra \emph{if and only if} it admits a bounded geometry, quasi-geodesic coarse median metric. Such a metric, when it exists, is unique up to quasi-isometry.
\end{thm}

\begin{proof}
For a bounded valency quasi-geodesic coarse median algebra $(X,\mu)$, Theorem \ref{unique metric prop} implies that the triple $(X,d_\mu,\mu)$ is a coarse median space of bounded geometry, where $d_\mu$ is the induced metric. Now Proposition \ref{quasi-geodesic} implies that $d_\mu$ is quasi-geodesic.

Conversely, for a bounded geometry quasi-geodesic coarse median metric $(X,d,\mu)$, Theorem \ref{uniquemetricthm} implies that $d$ is quasi-isometric to the induced metric $d_\mu$. Hence Lemma \ref{bdd geo} implies that $(X,\mu)$ has bounded valency and Proposition \ref{quasi-geodesic} implies that $(X,\mu)$ is quasi-geodesic.
\end{proof}

\subsection{The rank of a coarse median algebra}

Motivated by Theorem \ref{hyper rank}, we make the following definition.

\begin{defn}
A coarse median algebra $(X, \mu)$ is said to \emph{have rank at most $n$} if there is a non-decreasing function $\varphi: \Rp \to \Rp$ such that for any $x_1,\ldots,x_{n+1}$ and $p,q \in X$, we have
\begin{equation*}
 \min\{\sharp[p,\mu(x_i,p,q)]:i=1,\ldots,n+1\} \leqslant \varphi(\max\{\sharp[p,\langle x_i,x_j,p\rangle]: i\neq j\}).
\end{equation*}
\end{defn}

\begin{lem}\label{cma rank lemma}
The rank of a bounded valency coarse median algebra $(X,\mu)$ agrees with the rank of the corresponding coarse median space $(X,d_\mu,\mu)$ provided by Theorem \ref{unique metric prop}.
\end{lem}

\begin{proof}
Lemma \ref{finiteness} provides a non-decreasing function $C:\Rp\rightarrow \Rp$ such that
\[
d_\mu{}(a,b) < \sharp [a,b]\leq C(d_\mu{}(a,b)).
\]

If the coarse median algebra $(X,\mu)$ has rank at most $n$, then by definition there exists a non-decreasing $\varphi:\Rp \to \Rp$ such that for any $x_1,\ldots,x_{n+1}$ and $p,q \in X$,
\begin{align*}
  &\min\{d_\mu{}(p,\mu(x_i,p,q)):i=1,\ldots,n+1\}<\min\{\sharp[p,\mu(x_i,p,q)]:i=1,\ldots,n+1\} \\
  &\leqslant \varphi(\max\{\sharp[p,\langle{x_i,x_j,p}\rangle]: i\neq j\})
  \leqslant \varphi(\max\{C(d_\mu{}(p,\langle x_i,x_j,p\rangle)): i\neq j\})\\
  &=\varphi\circ C(\max\{d_\mu{}(p,\langle x_i,x_j,p\rangle): i\neq j\}).
\end{align*}
So by Theorem \ref{hyper rank} the coarse median space $(X,d_\mu, \mu)$ has rank at most $n$.

Conversely if the coarse median space $(X,d_\mu, \mu)$ has rank at most $n$, then by Theorem \ref{hyper rank}: There exists a non-decreasing $\varphi:\Rp\to\Rp$ such that for any $x_1,\ldots,x_{n+1}$ and $p,q \in X$,
      \begin{align*}
        &\min\{\sharp[p,\mu(x_i,p,q)]:i=1,\ldots,n+1\}\leq \min\{C(d_\mu{}(p,\mu(x_i,p,q))):i=1,\ldots,n+1\}\\
        &=C(\min\{d_\mu{}(p,\mu(x_i,p,q)):i=1,\ldots,n+1\})
        \leqslant C\circ\varphi(\max\{d_\mu{}(p,\langle x_i,x_j,p\rangle): i\neq j\})\\
        &\leqslant C\circ\varphi(\max\{\sharp[p,\langle x_i,x_j,p\rangle]: i\neq j\}).
      \end{align*}
So the coarse median algebra $(X,\mu)$ also has rank at most $n$.
\end{proof}

It is interesting to consider this in the context of spaces of rank 1 where we obtain a correspondence between quasi-geodesic, bounded valency coarse median algebras of rank $1$ and bounded geometry geodesic hyperbolic spaces up to quasi-isometry.  One direction is provided by \cite[Lemma 3.1]{bowditch2013coarse} and Theorem \ref{unique metric prop}.  For the converse we have:

\begin{thm}
Let $(X,\mu)$ be a bounded valency, quasi-geodesic coarse median algebra of rank $1$.  Then there exists a metric $d$ such that  $(X,d)$ is a geodesic hyperbolic metric space and its natural coarse median is uniformly close to $\mu$. 
\end{thm}

\begin{proof}
By Proposition \ref{quasi-geodesic}  there exists $C>0$ such that the complex $P_C(X, \nu)$ is connected and $d_\nu$ is bi-Lipschitz to the edge-path metric $d_{P_C}$ on the complex.  We take $d=d_{P_C}$.  Since this is geodesic \cite[Theorem 4.2]{niblo2017four} shows that (quasi)-geodesics in $P_C(X,\mu)$ are close to intervals, and hence by Theorem \ref{hyperbolic char} the slim triangle condition holds.

The natural coarse median of three points $a,b,c$ in this hyperbolic space is chosen from the intersection of $\delta$-neighbourhoods of the geodesics $\vec{ab},\vec{bc},\vec{ca}$ and is therefore in the intersection of $K$-neighbourhoods of the intervals $[a,b],[b,c],[c,a]$ for some fixed $K$.  This is a (uniformly) bounded set containing the original median, hence the new and original medians are uniformly close.
\end{proof}

\appendix
\section{A Categorical viewpoint}
To amplify and clarify the claim that coarse median spaces, coarse interval spaces and coarse median algebras are in some sense the same we will define suitable categories and show that they are equivalent.

\subsection{The coarse median (space) category}\label{The coarse median (space) category}

\begin{defn}\label{coarse median category}
Let $(X,d_X, \mu_X)$ and $(Y,d_Y, \mu_Y)$ be coarse median structures (see Definition \ref{coarse median operator}). A map $f: X \to Y$  is a \emph{$(\rho_+,C)$-coarse median map} if it is a $C$-quasi-morphism as well as a $\rho_+$-coarse map. As usual, we omit mentioning parameters unless we are keeping track of the values.
\end{defn}

\begin{rem}\label{composition of morphisms}
Note that without the assumption of coarseness for the map in this definition, it is not the case that coarse median maps compose to give coarse median maps. The issue is that while the coarse median of the three points $fg(a), fg(b), fg(c)$ is necessarily close to the image under $f$ of the coarse median of $g(a), g(b), g(c)$, without requiring $f$ to be coarse we cannot control the distance between this image and the image under $fg$ of the coarse median $\mu(a,b,c)$.
\end{rem}

Given a set $X$ and a metric space $(Y,d_Y)$  and functions  $f,g: X\rightarrow Y$ we will write $f \sim g$ if $f$ is $s$-close to $g$ for some $s\geq 0$. This is an equivalence relation and the equivalence class of $f$ is denoted by $[f]$. Applying this to  coarse median maps $\mu, \mu'$ on $X$ recovers the notion of uniform closeness discussed in Section 2.

\begin{defn}
The \emph{coarse median category}, denoted $\CM$, is defined as follows:
\begin{itemize}
  \item The objects are coarse median structures $(X,d_X,\mu_X)$;
  \item Given two objects $\mathcal{X}=(X,d_X,\mu_X)$ and $\mathcal{Y}=(Y,d_Y,\mu_Y)$ the morphism set is
$$\mor_{\CM}(\mathcal X,\mathcal Y):=\{\mbox{~coarse median maps~from~}X\mbox{~to~}Y~\}/\sim;$$
  \item Compositions are induced by compositions of maps.
\end{itemize}
The \emph{coarse median space category}, denoted $\CMS$, is the full subcategory whose objects are coarse median \emph{spaces}, i.e.,\ those whose coarse median additionally satisfies axioms (M1) and (M2).
\end{defn}

The objects of $\CM$ are those satisfying Bowditch's original definition \cite[Section 8]{bowditch2013coarse}. We now characterise categorical isomorphisms in a more practical way.

\begin{lem}\label{char for iso in CM}
Let $\mathcal X, \mathcal Y$ be objects in $\CM$ and $[f]\in \mor_\CM(\mathcal X, \mathcal Y)$. Then $[f]$ is an isomorphism in the category $\CM$ \emph{if and only if} $f$ is a coarse equivalence.
\end{lem}

\begin{proof}
Let $\mathcal{X}=(X,d_X,\mu_X)$ and $\mathcal{Y}=(Y,d_Y,\mu_Y)$. Suppose $[f]$ is an isomorphism in the category $\CM$, i.e., there exists another coarse median map $g: Y \to X$ such that $[f][g]=[\id_Y]$ and $[g][f]=[\id_X]$. Hence clearly, $f$ is a coarse equivalence.

On the other hand, suppose $f: X \to Y$ is a $(\rho_+,C)$-coarse median map as well as a $(\rho_+,C)$-coarse equivalence. In other words, there exists a $\rho_+$-coarse map $g: Y\to X$ such that $fg$ and $gf$ are $C$-close to the identities. It suffices to show that $g$ is a coarse median map. For any $x,y,z\in Y$, $fg \thicksim_C \id_Y$ implies that there exist $a,b,c\in X$ such that $f(a)\thicksim_C x$, $f(b)\thicksim_C y$ and $f(c)\thicksim_C z$. Since $g$ is $\rho_+$-bornologous, we have $gf(a) \thicksim_{\rho_+(C)}g(x)$, $gf(b) \thicksim_{\rho_+(C)}g(y)$ and $gf(c) \thicksim_{\rho_+(C)}g(z)$. Let $\rho_X,\rho_Y$ be the uniform bornology parameters of $\mathcal X,\mathcal Y$ provided by (C1). Then we have
$$\mu({g(x),g(y),g(z)})_X \thicksim_{\rho_X(3\rho_+(C))}\mu({gf(a),gf(b),gf(c)})_X \thicksim_{\rho_X(3C)}\mu(a,b,c)_X.$$
We also have
$$g(\mu(x,y,z)_Y)\thicksim_{\rho_+(\rho_Y(3C))}g(\mu({f(a),f(b),f(c)})_Y)\thicksim_{\rho_+(C)}gf(\mu(a,b,c)_X)
\thicksim_C \mu(a,b,c)_X.$$
Combining these, we have
$$\mu({g(x),g(y),g(z)})_X\thicksim_{C'}g(\mu(x,y,z)_Y)$$
for $C'=\rho_X(3\rho_+(C))+\rho_X(3C)+\rho_+(\rho_Y(3C))+\rho_+(C)+C$.
\end{proof}

\begin{rem}\label{dep of para for coarse median iso}
Recall from Definition \ref{cms isom} that a $(\rho_+,C)$-coarse median isomorphism $f$ is a $(\rho_+,C)$-coarse median map and a $(\rho_+,C)$-coarse equivalence. Hence the previous lemma states that such an $f$ is a $(\rho_+,C)$-coarse median isomorphism \emph{if and only if} it represents a categorical isomorphism.
Any $(\rho_+,C)$-coarse inverse $g$ for $f$ is a $(\rho_+,C')$-coarse median isomorphism with the constant $C'$ depending only on $\rho_+,C$ and parameters of $X,Y$. And in this case, $[g]$ is a categorical inverse of $[f]$.
\end{rem}

We now discuss the relationship between the categories of coarse median spaces $\CMS$ and coarse median structures $\CM$.
\begin{prop}\label{cms equiv cm}
The inclusion functor $\iota_\mathcal{M}:\CMS \hookrightarrow \CM$ gives an equivalence of categories.
\end{prop}

\begin{proof}
As $\CMS$ is a full subcategory of $\CM$, it suffices to show that each object in $\CM$ is isomorphic to an object of $\CMS$ (see for example \cite[Theorem 1, Section IV.4]{mac2013categories} or \cite[Proposition 1.3, Chapter 1]{jacobson2012basic}). For $(X,d,\mu)$ an object in $\CM$, as remarked before Definition \ref{coarse median operator}, $\mu$ is uniformly close to another coarse median $\mu'$ satisfying (M1) and (M2). The identity map $\mathrm{Id}_X$ is then a coarse median map from $(X,d,\mu')$ to $(X,d,\mu)$ which provides the required isomorphism in $\CM$.
\end{proof}

\subsection{The coarse interval (space) category}

We will define the coarse interval category and the coarse interval space category in this subsection. As we did in the coarse median case, let us start with morphisms.

\begin{defn}
Let $(X,d_X, [\cdot, \cdot]_X)$ and $(Y,d_Y,[\cdot, \cdot]_Y)$ be  coarse interval structures (see Definition \ref{def: coarse interval structure}). A map $f: X \to Y$ is said to be a \emph{$(\rho_+,C)$-coarse interval map} if $f$ is a $\rho_+$-coarse map and for any $a,b\in X$, $f([a,b]) \subseteq \N_C([f(a),f(b)])$. As usual, we omit parameters unless they are required.

Given coarse interval maps $f,g$ from $X$ to $Y$, we introduce the notation $f \sim g$ if $f$ is $s$-close to $g$ for some $s$. This is an equivalence relation and we denote the equivalence class of $f$ by $[f]$.
\end{defn}

\begin{defn}
The \emph{coarse interval category}, denoted $\CI$\!, is defined as follows:
\begin{itemize}
  \item  The objects are coarse interval structures $(X,d_X,[\cdot,\cdot]_X)$;
  \item Given two objects: $\mathcal{X}=(X,d_X,[\cdot,\cdot]_X)$ and $\mathcal{Y}=(Y,d_Y,[\cdot,\cdot]_Y)$, the morphism set is
  $$
  \mor_{\CI}(\mathcal X,\mathcal Y):=\{\mbox{~coarse~interval~maps~from~}X\mbox{~to~}Y~\}/\sim;
  $$
  \item Compositions are induced by compositions of maps.
\end{itemize}
The \emph{coarse interval space category}, denoted $\CIS$, is the full subcategory whose objects are coarse interval \emph{spaces}, i.e., those satisfying the stronger axioms (I1)$\sim$(I3).
\end{defn}

As in Lemma \ref{char for iso in CM}, we can characterise categorical isomorphisms in a more practical way. Let us start with the following observation:

\begin{lem}\label{hausdorff control}
Let $(X,d_X,[\cdot,\cdot]_X),(Y,d_Y,[\cdot,\cdot]_Y)$ be  coarse interval structures and $f:X \to Y$ be a coarse interval map as well as a coarse equivalence. Then there exists some constant $D>0$ such that for any $a,b\in X$,
$$d_H(f([a,b]),[f(a),f(b)])\leqslant D.$$
\end{lem}

\begin{proof}
Suppose $f$ is a $(\rho_+,C)$-coarse interval map with $C\geqslant 3\kappao$ where $\kappao$ is the parameter of $[\cdot,\cdot]_Y$ given in axioms (I1)' and (I3)' and $g:Y\to X$ is a $\rho_+$-bornologous map such that $f\circ g\thicksim_C \id_{Y}$ and $g\circ f \thicksim_C \id_{X}$. For $z\in [f(a),f(b)]$ and for $c=g(z)$, we have $f(c)\thicksim_C z$. Hence by Remark \ref{ends close to intervals} as $C\geqslant 3\kappao$, we have
$$f(c) \in \N_C([f(a),f(b)]) \cap \N_C([f(b),f(c)]) \cap \N_C([f(c),f(a)]).$$

On the other hand, since $f$ is a $(\rho_+,C)$-coarse interval map, we have
\begin{eqnarray*}
  f([a,b]\cap [b,c]\cap [c,a]) &\subseteq& f([a,b])\cap f([b,c]) \cap f([c,a]) \\
   &\subseteq& \N_C([f(a),f(b)]) \cap \N_C([f(b),f(c)]) \cap \N_C([f(c),f(a)]).
\end{eqnarray*}
This has diameter at most $C'$ for some constant $C'$ by axiom (I3)'. Hence there exists $c'\in [a,b]$ such that $f(c)\thicksim_{C'}f(c')$ which implies that $z\thicksim_C f(c)\thicksim_{C'}f(c')$, i.e., $z\in \N_{C+C'}(f([a,b]))$. Taking $D=C+C'$ we have $d_H(f([a,b]),[f(a),f(b)])\leqslant D$ as required.
\end{proof}

Now we give a characterisation of categorical isomorphism in $\CI$ and $\CIS$.

\begin{lem}\label{char for iso in CI}
Let $(X,d_X,[\cdot,\cdot]_X)$ and $(Y,d_Y,[\cdot,\cdot]_Y)$ be two coarse interval structures and $f:X \to Y$ be a coarse interval map. Then $[f]$ is an isomorphism in $\CI$ \emph{if and only if} $f$ is a coarse equivalence. The same holds in $\CIS$ by restricting to this full subcategory.
\end{lem}

\begin{proof}
Suppose $[f]$ is an isomorphism in $\CI$, i.e., there exists another coarse interval map $g: Y \to X$ such that $[f][g]=[\id_Y]$ and $[g][f]=[\id_X]$. Hence clearly, $f$ is a coarse equivalence.

On the other hand, suppose $f$ is a $(\rho_+,C)$-interval morphism and $g:Y\to X$ is $\rho_+$-coarse such that $fg\thicksim_C \id_{Y}$ and $gf \thicksim_C \id_{X}$. It suffices to show that there exists some constant $C'>0$ such that for any $z,w\in Y$, $g([z,w]) \subseteq \N_{C'}(g(z),g(w))$. Since $fg\thicksim_C \id_{Y}$, we have $z \thicksim_C f(z')$ and $w \thicksim_C f(w')$ for $z'=g(z)$ and $w'=g(w)$. By axioms (I1)' and (I2), there exists some constant $K>0$ such that $[z,w] \subseteq \mathcal N_K([f(z'),f(w')]$. Hence
$$g([z,w]) \subseteq g(\mathcal N_K([f(z'),f(w')])) \subseteq \N_{\rho_+(K)}( g([f(z'),f(w')]) ).$$
By Lemma \ref{hausdorff control}, there exists a constant $D>0$ such that $[f(z'),f(w')]\subseteq \N_D(f[z',w'])$, which implies that
\begin{eqnarray*}
  g([z,w]) &\subseteq& \N_{\rho_+(K)}( g([f(z'),f(w')]) ) \subseteq \N_{\rho_+(K)}(g(\N_D(f([z',w'])))) \\
   &\subseteq& \N_{\rho_+(K)+\rho_+(D)}(gf([z',w'])) \subseteq \N_{C'}([z',w'])=\N_{C'}([g(z),g(w)]),
\end{eqnarray*}
where $C'=\rho_+(K)+\rho_+(D)+C$ depends only on $\rho_+,C$ and parameters of $[\cdot,\cdot]_X$ and $[\cdot,\cdot]_Y$.
\end{proof}

\begin{prop}\label{ci equiv cis}
The inclusion functor $\iota_{\I}:\CIS \hookrightarrow \CI$ gives an equivalence of categories.
\end{prop}

\begin{proof}
This follows from Lemma \ref{coarse interval close}. The argument is similar to the proof of  Proposition \ref{cms equiv cm}, hence omitted.
\end{proof}

\subsection{Equivalence of the coarse median and coarse interval categories}
Now we construct functors connecting categories $\CM(\mathcal{S})$ and $\CI(\mathcal{S})$, and show that these functors give equivalences of categories. Theorem \ref{induce cm and ci} (1) offers a functor from $\CM$ to $\CI$ as follows:

\begin{lem}\label{morphisms}
Let $(X,d_X,\mu_X)$ and $(Y,d_Y,\mu_Y)$ be objects in the category $\CM$ and $f:X\rightarrow Y$ be a $(\rho_+,C)$-coarse median map. Suppose $(X,d_X,[\cdot,\cdot]_X)$ and $(Y,d_Y,[\cdot,\cdot]_Y)$ are the induced coarse interval structures. Then $f$ is a $(\rho_+,C)$-coarse interval map from $(X, d_X, [\cdot,\cdot]_X)$ to $(Y, d_Y, [\cdot,\cdot]_Y)$.
\end{lem}

\begin{proof}
For any $x,y,z\in X$ we have $f(\mu(x,y,z)_X) \thicksim_C \mu({f(x),f(y),f(z)})_Y$. Hence for $\mu(x,y,z)_X \in [x,y]$ we have $f(\mu(x,y,z)_X) \in \N_C([f(x),f(y)])$. So $f([x,y]) \subseteq \N_C([f(x),f(y)])$ which completes the proof.
\end{proof}

\begin{defn}\label{functor F}
We define a functor $F:\CM\rightarrow \CI$ by setting $F(X,d_X,\mu_X)$ to be the induced coarse interval structure $(X,d_X,[\cdot,\cdot]_X)$ and defining $F[f]=[f]$ on morphisms. This is well defined by Lemma \ref{morphisms} and also restricts to give a functor $F_{\mathcal{S}}: \CMS \rightarrow \CIS$ by Proposition \ref{coarse interval}.
\end{defn}

Now we consider the opposite direction. Theorem \ref{induce cm and ci} (2) provides a functor from $\CI$ to $\CM$ as follows:

\begin{lem}\label{morphisms2}
Let $(X,d_X,[\cdot,\cdot]_X)$ and $(Y,d_Y,[\cdot,\cdot]_Y)$ be objects in the category $\CI$ and let $f:X\rightarrow Y$ be a $(\rho_+,C)$-coarse interval map. Suppose $(X,d_X,\mu_X)$ and $(Y,d_Y,\mu_Y)$ are any induced coarse median structures. Then $f$ is a $(\rho_+, \g(\rho_+(\kappao)+C))$-coarse median map from $(X, d_X, \mu_X)$ to $(Y, d_Y, \mu_Y)$, where $\kappao$ is the parameter in axiom (I1)' for $(X,d_X,[\cdot,\cdot]_X)$ and $\g$ is the parameter in axiom (I3)' for $(Y,d_Y,[\cdot,\cdot]_Y)$.
\end{lem}

\begin{proof}
By definition $f([x,y]) \subseteq \N_C([f(x),f(y)])$ for any $x,y\in X$. Now we have:
\begin{eqnarray*}
  f(\mu(a,b,c)_X) &\in & f(\N_{\kappao}([a,b])\cap \N_{\kappao}([b,c])\cap \N_{\kappao}([c,a]) ) \\
  &\subseteq& \N_{\rho_+(\kappao)}(f([a,b])) \cap \N_{\rho_+(\kappao)}(f([b,c])) \cap \N_{\rho_+(\kappao)}(f([c,a])) \\
   &\subseteq & \N_{C'}([f(a),f(b)]) \cap \N_{C'}([f(b),f(c)]) \cap \N_{C'}([f(c),f(a)]) \\
   &\subseteq & B_{\g(C')}(\mu({f(a),f(b),f(c)})_Y)
\end{eqnarray*}
for $C'=\rho_+(\kappao)+C$ and any $a,b,c \in X$. Hence we have
$$f(\mu(a,b,c)_X) \thicksim_{\g(C')} \langle f(a),f(b),f(c)\rangle_Y,$$
which implies that $f$ is a $(\rho_+, \g(\rho_+(\kappao)+C))$-coarse median map.
\end{proof}

\begin{defn}\label{functor G}
We define a functor $G:\CI\rightarrow \CM$ by setting $G(X,d_X,[\cdot,\cdot]_X)=(X,d_X,\mu_X)$, where $\mu_X$ is some (chosen) induced coarse median on $X$. The choice here is well defined up to equivalence of coarse medians and we  define $G[f]=[f]$ on morphisms. This definition makes sense by Theorem \ref{induce cm and ci} and Lemma \ref{morphisms2} and restricts to give a functor $G_{\mathcal{S}}: \CIS \rightarrow \CMS$ by Theorem \ref{coarse interval converse}.
\end{defn}

\begin{thm}\label{cat equiv}
The functors $F$ and $G$ from Definitions \ref{functor F}, \ref{functor G} provide an equivalence of categories between coarse median structures $\CM$ and coarse interval structures $\CI$.  This equivalence restricts to give an equivalence of categories between coarse median spaces $\CMS$ and coarse interval spaces $\CIS$.
\end{thm}

\begin{proof}
We are required to show that $G\circ F$ is naturally isomorphic to $\id_{\CM}$ and that $F\circ G$ is naturally isomorphic to $\id_{\CI}$.

\textbf{(1).} First consider $G\circ F$. Given a metric space $(X,d_X)$ with a coarse median $\mu_X$, by definition $F(X,d_X,\mu_X)$  is the induced coarse interval structure $(X,d_X,[\cdot,\cdot]_X)$. Now apply $G$ to the triple $(X,d_X,[\cdot,\cdot]_X)$ and denote the chosen induced operator by $\mu'_X$. It follows directly from Theorem \ref{induce cm and ci}(3) that $\mu'_X$ and $\mu_X$ are uniformly close. Hence the identity $\id_X: (X,d_X,\mu_X) \to (X,d_X,\mu'_X)$ gives an isomorphism in the category $\CM$. Furthermore, for any coarse median structure $(Y,d_Y, \mu_Y)$ and coarse median map $f: X \rightarrow Y$, the following diagram commutes since $G,F$ do not change the morphisms:
$$\xymatrixcolsep{1.8cm}\xymatrixrowsep{1.8cm}\xymatrix{
  (X,d_X,\mu_X) \ar[r]^-{\textstyle [\id_X]} \ar[d]^{\textstyle \id_\CM([f])} & G\circ F(X,d_X,\mu_X)=(X,d_X,\mu'_X) \ar[d]^{\textstyle G\circ F([f])}  \\
  (Y,d_Y,\mu_Y) \ar[r]^-{\textstyle [\id_Y]} & G\circ F(Y,d_Y,\mu_Y)=(Y,d_Y,\mu'_Y).
  }
$$
Hence the natural transformation $(X,d_X,\mu_X) \mapsto [\id_X]$ gives a natural isomorphism from $\id_{CM}$ to $G\circ F$. This restricts to give a natural isomorphism from $\id_\CMS$ to $G_{\mathcal S}\circ F_\mathcal{S}$.

\textbf{(2).} Next consider $F\circ G$. Given a coarse interval structure $(X,d_X,[\cdot,\cdot]_X)$, we have $G(X,d_X,[\cdot,\cdot]_X)=(X,d_X,\mu_X)$ where $\mu_X$ is the chosen induced coarse median operator on $X$. Apply $F$ to the coarse median structure $(X,d_X,\mu_X)$ and denote the induced interval structure by $(X, d_X, [\cdot,\cdot]_X')$. It follows directly from Theorem \ref{induce cm and ci}(3) that $[\cdot,\cdot]_X$ and $[\cdot,\cdot]'_X$ are uniformly close. Therefore, the identity $\id_X: (X,d_X,[\cdot,\cdot]_X) \to (X,d_X,[\cdot,\cdot]'_X)$ gives an isomorphism in the category $\CI$. Furthermore, for any other coarse interval structure  $(Y,d_Y,[\cdot,\cdot]_Y)$ and  coarse interval map $f: X \rightarrow Y$, the following diagram again clearly commutes:
$$\xymatrixcolsep{1.8cm}\xymatrixrowsep{1.8cm}\xymatrix{
  (X,d_X,[\cdot,\cdot]_X) \ar[r]^-{\textstyle [\id_X]} \ar[d]^{\textstyle \id_\CI([f])} & F\circ G(X,d_X,[\cdot,\cdot]_X)=(X,d_X,[\cdot,\cdot]'_X) \ar[d]^{\textstyle F\circ G([f])}  \\
  (Y,d_Y,[\cdot,\cdot]_Y) \ar[r]^-{\textstyle [\id_Y]} & F\circ G(Y,d_Y,[\cdot,\cdot]_Y)=(Y,d_Y,[\cdot,\cdot]'_Y).
  }
$$
Hence the natural transformation $(X,d_X,[\cdot,\cdot]_X) \mapsto [\id_X]$ gives a natural isomorphism from $\id_\CI$ to $F\circ G$. As usual this restricts to give a natural isomorphism from $\id_\CIS$ to $F_{\mathcal S}\circ G_\mathcal{S}$.
\end{proof}

Combining Propositions \ref{cms equiv cm}, \ref{ci equiv cis}, Theorem \ref{cat equiv} and Corollary \ref{rank preserving}, we obtain the following.

\begin{thm}\label{cat equiv final}
Consider the following diagram:
\begin{displaymath}
\xymatrix@=1.25cm{
  \CM \ar@/^/[r]^{\textstyle F} & \CI \ar@/^/[l]^{\textstyle G}  \\
  \CMS \ar[u]^{\textstyle\iota_{\mathcal{M}}} \ar@/^/[r]^{\textstyle F_{\mathcal{S}}}
                & \CIS.  \ar[u]_{\textstyle\iota_{\mathcal{I}}} \ar@/^/[l]^{\textstyle G_{\mathcal{S}}}
  }
\end{displaymath}
We have:
\begin{itemize}
  \item $F \circ \iota_{\mathcal{M}} = \iota_{\mathcal{I}} \circ F_{\mathcal{S}}$;
  \item $\iota_{\mathcal{M}} \circ G_{\mathcal{S}} = G \circ \iota_{\mathcal{I}}$;
  \item $\iota_{\mathcal{M}}$ gives an equivalence of categories between $\CMS$ and $\CM$;
  \item $\iota_{\mathcal{I}}$ gives an equivalence of categories between $\CIS$ and $\CI$;
  \item $(F,G)$ gives an equivalence of categories between $\CM$ and $\CI$;
  \item $(F_{\mathcal{S}},G_{\mathcal{S}})$ gives an equivalence of categories between $\CMS$ and $\CIS$.
\end{itemize}
Furthermore all of these functors preserve rank in the sense of coarse median structures and coarse interval structures.
\end{thm}

\begin{rem}
We finally note that one can restrict the allowed metric spaces to quasi-geodesic spaces.  In this case the above equivalences of categories restrict to equivalences between the full subcategories of quasi-geodesic coarse median spaces and quasi-geodesic coarse interval spaces.
\end{rem}

\subsection{Comparing the categories of coarse median algebras and coarse median spaces}

In the spirit of Section \ref{The coarse median (space) category}, we now consider the category of bounded valency coarse median algebras. By analogy with the notion of coarse median map, we define a \emph{coarse median algebra homomorphism} from $(X, \mu_X)$ to $(Y, \mu_Y)$ to be a function $f:X\to Y$ such that

\begin{enumerate}
\item there exist a  constant $C$ such that that for all $a,b,c\in X$,
\[
\sharp[\mu({f(a),f(b),f(c)})_Y, f(\mu(a,b,c)_X)]_Y\leq C;
\]
\item there exists a non-decreasing function $\rho:\Rp\rightarrow \Rp$ such that for all $a,b\in X$,
\[
\sharp[f(a), f(b)]\leq \rho(\sharp[a,b]);
\]
\item $f$ is finite-to-$1$.
\end{enumerate}

Where we need to keep track of $C, \rho$ we will refer to $f$ as a $(C,\rho)$-coarse median algebra homomorphism.

Note that the requirement of ``finite-to-$1$" is analogous to properness in the definition of coarse  map. When $C$ can be taken to be $1$ then $\mu({f(a),f(b),f(c)})_Y= f(\mu(a,b,c)_X)$ and $f$ is a homomorphism of ternary algebras.  In particular if $X$ and $Y$ are median algebras and $C=1$ then $f$ is a homomorphism of median algebras and the second and third conditions require that $f$ is also a coarse map in the geometric sense. From the algebraic point of view one would not expect the latter conditions to be required, however without them the composition of coarse median algebra homomorphisms would not in general yield another coarse median algebra homomorphism (cf. Remark \ref{composition of morphisms}).

\begin{lem}\label{lem:composition of coarse median algebra morphisms}
Let $f: (X,\mu_X) \rightarrow (Y,\mu_Y)$ and $g: (Y,\mu_Y) \rightarrow (Z,\mu_Z)$ be coarse median algebra homomorphisms between bounded valency coarse median algebras $(X,\mu_X), (Y,\mu_Y)$ and $(Z,\mu_Z)$. Then the composition $gf:(X,\mu_X) \rightarrow (Z,\mu_Z)$ is a coarse median algebra homomorphism as well.
\end{lem}

\begin{proof}
Let $f$  be a  $(C_f, \rho_f)$-coarse median algebra homomorphism and $g$   a  $(C_g, \rho_g)$-coarse median algebra homomorphism. Since $(Z,\mu_Z)$ has bounded valency, Lemma \ref{bdd geo} and Theorem \ref{unique metric prop} imply that the induced metric $d_{\mu_Y}$ has bounded geometry and $(Z,d_{\mu_Z}, \mu_Z)$ is a coarse median space. From Lemma \ref{finiteness} there exists a non-decreasing function $\phi_Z:\Rp\rightarrow \Rp$ such that for any $u,v\in Z$ we have
\[
d_{\mu_Z}{}(u,v) < \sharp [u,v]_Z\leq \phi_Z(d_{\mu_Z}{}(u,v)).
\]
Now for any $a,b,c\in X$ we have:
\begin{align*}
d_{\mu_Z}&(g(\langle f(a), f(b),f(c) \rangle_Y), gf(\mu(a,b,c)_X)) \leq \sharp [g(\langle f(a), f(b),f(c) \rangle_Y), gf(\mu(a,b,c)_X)]_Z \\
&\leq \rho_g(\sharp [\langle f(a), f(b),f(c) \rangle_Y, f(\mu(a,b,c)_X)]_Y) \leq \rho_g(C_f),
\end{align*}
and
\begin{align*}
&d_{\mu_Z}(g(\langle f(a), f(b),f(c) \rangle_Y), \langle gf(a), gf(b), gf(c) \rangle_Z) \\
&\leq \sharp [g(\langle f(a), f(b),f(c) \rangle_Y), \langle gf(a), gf(b), gf(c) \rangle_Z]_Z \\
&\leq C_g.
\end{align*}
Hence
\begin{align*}
&\sharp[gf(\mu(a,b,c)_X), \langle gf(a), gf(b), gf(c) \rangle_Z]_Z \\
&\leq \phi_Z(d_{\mu_Z}(gf(\mu(a,b,c)_X), \langle gf(a), gf(b), gf(c) \rangle_Z)) \\
&\leq \phi_Z(\rho_g(C_f) + C_g).
\end{align*}
Thus condition (1) holds for $gf$. As for condition (2), we have
\[
\sharp [gf(a),gf(b)]_Z \leq \rho_g(\sharp [f(a),f(b)]_Y) \leq \rho_g\rho_f([a,b]_X).
\]
Finally, since  $f,g$ are finite-to-$1$, their composition $gf$ is finite-to-$1$ as well so $gf$ is  a coarse median algebra homomorphism.
\end{proof}

Two coarse median algebra homomorphisms $f,g$ are said to be \emph{equivalent}, denoted by $f \sim g$, if there is a constant $D$ such that for all $x\in X$, $\sharp[ f(x), g(x)]_Y\leq D$. Now Lemma \ref{lem:composition of coarse median algebra morphisms} allows us to make the following definition:

\begin{defn}
The \emph{bounded valency coarse median algebra category}, denoted by $\mathcal{BCMA}$, is defined as follows:
\begin{itemize}
  \item  The objects are coarse median algebras $(X,\mu_X)$ of bounded valency;
  \item Given two objects: $\mathcal{X}=(X,\mu_X)$ and $\mathcal{Y}=(Y,\mu_Y)$, the morphism set is
  $$
  \{\mbox{~coarse~median~algebra~homomorphisms~from~}X\mbox{~to~}Y~\}/\sim;
  $$
  \item Compositions are induced by compositions of homomorphisms.
\end{itemize}
\end{defn}

Let $\mathcal{BCMS}$ denote the full subcategory of $\CMS$ whose objects are bounded geometry coarse median spaces. We will construct a functor $H: \mathcal{BCMA} \rightarrow \mathcal{BCMS}$ given by equipping each coarse median algebra $(X, \mu_X)$ with its induced metric.

\begin{lem}\label{lem: functor H}
Defining $H([f]) = [f]$ makes this a functor.
\end{lem}

\begin{proof}
We first show that for bounded valency coarse median algebras $(X,\mu_X)$, $(Y,\mu_Y)$ and a $(C_f, \rho_f)$-coarse median algebra homomorphism $f: X \to Y$, $f$ is also a coarse median map between $(X,d_{\mu_X},\mu_X)$ and $(Y,d_{\mu_Y},\mu_Y)$.

Clearly for any $a,b,c\in X$, the distance between $\langle f(a),f(b),f(c) \rangle_Y$ and $f(\mu(a,b,c)_X)$ is bounded by the cardinality of the associated interval, which is uniformly bounded by $C_f$. Hence $f$ is a $C_f$-quasi-morphism. As in the proof of Lemma \ref{lem:composition of coarse median algebra morphisms}, $(X,\mu_X)$ having bounded valency implies that there exists a non-decreasing function $\phi_X:\Rp\rightarrow \Rp$ such that for any $u,v\in X$, we have
\[
d_{\mu_X}{}(u,v) < \sharp [u,v]_X\leq \phi_X(d_{\mu_X}{}(u,v)).
\]
This implies
\[
d_{\mu_Y}(f(u),f(v)) \leq \sharp[f(u),f(v)]_Y \leq \rho_f([u,v]_X) \leq \rho_f\circ\phi_X(d(u,v)).
\]
Hence $f$ is $\rho_f\circ\phi_X$-bornologous. Bounded geometry and the fact that $f$ is finite-to-$1$ imply that $f$ is proper. Therefore, $f$ is a coarse median map.

Now suppose that $g:X\rightarrow Y$ is a coarse median algebra homomorphism which is equivalent to $f$. Now consider $d_{\mu_Y}(f(x), g(x))\leq \card[f(x), g(x)]_Y$. By assumption the latter is bounded hence $f$ is close to $g$.
\end{proof}

While the forgetful map which converts a bounded valency, bounded geometry coarse median space to the underlying coarse median algebra is a left inverse to $H$, this is not in general functorial.

\begin{figure}[h]
\begin{center}
\includegraphics[scale=0.6]{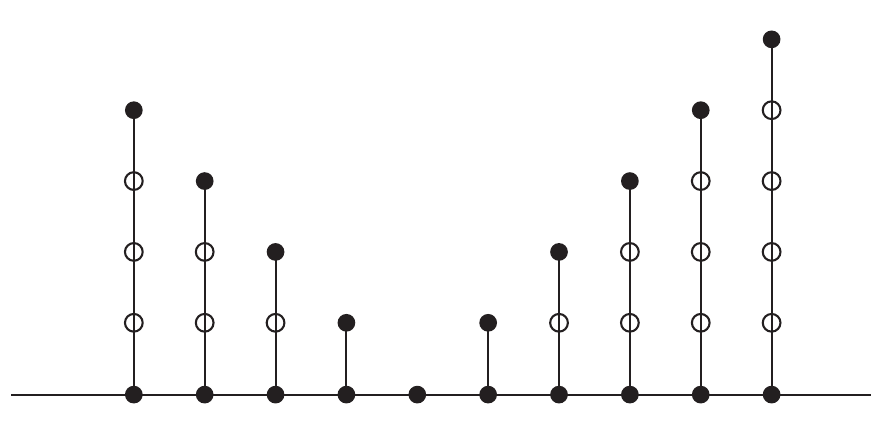}
\caption{The tree $T$ with the subspace $X$ identified by the solid vertices}
\label{tree_example}
\end{center}
\end{figure}

\begin{ex}\label{ex for non-functor}
Consider the tree $T$ obtained from $\mathbb Z$ by adding a spike of length $|n|$  to each integer $n$. All edges are taken to have length $1$. As a tree this is naturally a discrete median space and can be viewed as a coarse median space with its natural path metric. Now take the subspace $X$ consisting of the original points of $\mathbb Z$, together with the leaves of the tree, and equip this with the subspace metric (see Figure \ref{tree_example}). This is a median subalgebra and the inclusion is a morphism of coarse median spaces. However it is not a morphism of coarse median algebras, since taking $a$ to be the leaf on the spike based at the integer $b$ the interval $[a,b]_X$ has cardinality $2$, while its image in $T$ has cardinality $|b|+1$ contravening the second condition. Once again this illustrates that it is possible to endow a coarse median algebra with a metric which does not fully respect the algebraic structure. However restricting to the quasi-geodesic world, or more generally, imposing the induced metric prevents these problems and makes the forgetful map functorial.
\end{ex}

Just as CAT(0) cube complexes can be studied combinatorially as median algebras, we can apply Theorem \ref{unique metric prop 2} to obtain the following theorem showing that coarse median spaces can be studied as coarse median algebras.

\newcommand{\uniquemetricthm}{The forgetful functor, together with the induced metric functor provide an equivalence of categories between the category of \textbf{bounded geometry, quasi-geodesic coarse median spaces} and the category of  \textbf{bounded valency, quasi-geodesic coarse median algebras}, and this equivalence preserves rank.
}

\begin{thm}
The forgetful functor, together with the ``induced metric'' functor provide an equivalence of categories from \emph{bounded geometry quasi-geodesic coarse median spaces} to \emph{bounded valency quasi-geodesic coarse median algebras}, and this equivalence preserves rank.
\end{thm}

\begin{proof}
First let us show that the forgetful map is indeed a functor in this case. Let $(X,\mu_X,d_X)$ and $(Y,\mu_Y,d_Y)$ be two bounded geometry quasi-geodesic coarse median spaces, and $f:X \rightarrow Y$ be a coarse median map. Theorem \ref{unique metric prop 2} implies that $(X,\mu_X)$ and $(Y,\mu_Y)$ are bounded valency quasi-geodesic coarse median algebras and for some $L',C'>0$ the induced metrics $d_{\mu_X}, d_{\mu_Y}$ are $(L',C')$-quasi-isometric to $d_X, d_Y$, respectively.

Since $f$ is a quasi-morphism there exists $C>0$ such that for any $a,b,c \in X$
\[
\mu({f(a),f(b),f(c)})_Y \thicksim_C f\bigl(\mu(a,b,c)_X\bigr).
\]
Since $(Y,d_Y)$ has bounded valency, Lemma \ref{finiteness} provides a non-decreasing function $\phi_Y:\Rp\rightarrow \Rp$ such that for any $u,v\in Y$, we have $\sharp [u,v]_Y\leq \phi_Y(d_Y(u,v))$. Thus we obtain:
\[
\sharp [\mu({f(a),f(b),f(c)})_Y, f(\mu(a,b,c)_X)] \leq \phi_Y(C).
\]
On the other hand, assume that $f$ is $\rho_f$-coarse, then we have:
\begin{align*}
\sharp [f(a),f(b)]_Y &\leq \phi_Y(d_Y(f(a),f(b))) \leq \phi_Y \circ \rho_f (d_X(a,b)) \leq \phi_Y \circ \rho_f (L'd_{\mu_X}(a,b)+C')\\
&\leq \phi_Y \circ \rho_f (L'\sharp[a,b]_X+C').
\end{align*}
Note that we use the fact that $d_X$ and $d_{\mu_X}$ are $(L',C')$-quasi-isometric in the third inequality, which fails in Example \ref{ex for non-functor}, and this is the only place we need the condition of quasi-geodesity. Finally properness and bounded geometry imply that $f$ is finite-to-$1$. In conclusion, we have shown that $f$ is a coarse median algebra homomorphism which implies that the forgetful map is indeed a functor.

The rest of the statement follows directly from Theorem \ref{unique metric prop 2} and Lemma \ref{cma rank lemma}.
\end{proof}

Finally we note that while the definition of a coarse median space requires affine control (in axiom (C1)) the morphisms in the category $\CMS$ are only required to be coarse, not large-scale Lipschitz.  As mentioned in \cite{niblo2017four}, this suggests a generalisation of the notion of coarse median to allow bornologous control:  replace  the affine control axiom (C1) with the requirement that the map $a\mapsto \mu(a,b,c)$ is bornologous uniformly in $b,c$. Many of the results of this paper should carry over to this context.

However this does not greatly extend the class of examples for the following reasons.  Firstly in the quasi-geodesic case this is not a generalisation as bornologous control implies affine control.  More generally suppose that $(X,d,\mu)$ is a triple satisfying the generalised bornologous version of (C1) along with (M1), (M2) and (C2). If this has bounded geometry then the proof of Theorem \ref{unique metric prop} (3) $\implies$ (1) remains valid to show that $(X,\mu)$ is a coarse median algebra.  If moreover we have bounded valency then by Theorem \ref{unique metric prop}, $(X,d_\mu,\mu)$ is a coarse median space in the usual (affine) sense.  Thus, in the context of bounded geometry bounded valency spaces, the metric can always be adjusted to ensure affine control.

\bibliographystyle{plain}
\bibliography{bibfileCMA}

\end{document}